\documentclass[11pt]{amsart}

\usepackage{amsmath}
\usepackage{amssymb}

\usepackage{amsfonts}
\usepackage{amsthm}
\usepackage{enumerate}
\usepackage[mathscr]{eucal}

\usepackage{bbm}
\usepackage{graphics, graphicx}
\usepackage{verbatim}
\usepackage{tikz}

\newcommand{\ci}[1]{_{{}_{\!\scriptstyle{#1}}}}

\newcommand{\Be}{\begin{equation}}
\newcommand{\Ee}{\end{equation}}

\newcommand{\Bea}{\begin{eqnarray}}
\newcommand{\Eea}{\end{eqnarray}}
\newcommand{\Beas}{\begin{eqnarray*}}
\newcommand{\Eeas}{\end{eqnarray*}}
\newcommand{\Benu}{\begin{enumerate}}
\newcommand{\Eenu}{\end{enumerate}}
\newcommand{\Bi}{\begin{itemize}}
\newcommand{\Ei}{\end{itemize}}

\def\dyad{{\text{dyad}}}

\def\intslash{\rlap{\kern  .32em $\mspace {.5mu}\backslash$ }\int}
\def\qsl{{\rlap{\kern  .32em $\mspace {.5mu}\backslash$ }\int_{Q_x}}}

\def\emph#1{{\it #1 }}

\def\HF{{\text{\rm HF}}}

\def\ga{\gamma}

\def\Ups{\Upsilon}

\def\cf{{\it cf}}

\def\supp{{\text{\rm supp}}}

\def\inn#1#2{\langle#1,#2\rangle}

\def\meas{{\text{\rm meas}}}

\def\lc{\lesssim}
\def\gc{\gtrsim}

\def\eps{\varepsilon}

\def\ka{\kappa}
            
             \def\La{\Lambda}

\def\fA{{\mathfrak {A}}}

\def\fK{{\mathfrak {K}}}
\def\fL{{\mathfrak {L}}}
\def\fM{{\mathfrak {M}}}

\def\fS{{\mathfrak {S}}}

\def\bbC{{\mathbb {C}}}

\def\bbN{{\mathbb {N}}}

\def\bbR{{\mathbb {R}}}

\def\bbZ{{\mathbb {Z}}}

\def\sH{{\mathscr {H}}}

\def\cD{{\mathcal {D}}}
\def\cE{{\mathcal {E}}}

\def\cG{{\mathcal {G}}}

\def\cM{{\mathcal {M}}}

\def\cP{{\mathcal {P}}}

\def\cS{{\mathcal {S}}}
\def\cT{{\mathcal {T}}}

\def\cZ{{\mathcal {Z}}}

 at 10 true pt

\def\be#1{\begin{equation}\label{ #1}}
\def\endeq{\end{equation}}
\def\endal{\end{align}}
\def\bas{\begin{align*}}
\def\eas{\end{align*}}
\def\bi{\begin{itemize}}
\def\ei{\end{itemize}}

\def\eps{\varepsilon}

\def\emph#1{{\it #1}}
\def\textbf#1{{\bf #1}}

\usepackage{bbm}
\def\bbone{{\mathbbm 1}}

\theoremstyle{plain}
  \newtheorem{theorem}{Theorem}[section]

   \newtheorem{proposition}[theorem]{Proposition}
   
   \newtheorem{lemma}[theorem]{Lemma}
   \newtheorem{corollary}[theorem]{Corollary}

\theoremstyle{remark}

\theoremstyle{definition}

\begin{document}

\title[Haar projection numbers
and unconditional convergence]{Haar projection numbers
and failure of unconditional convergence in Sobolev spaces}

\author{Andreas Seeger \ \ \ \ \ \ \ \ \ Tino Ullrich}

\address{Andreas Seeger \\ Department of Mathematics \\ University of Wisconsin \\480 Lincoln Drive\\ Madison, WI, 53706, USA} \email{seeger@math.wisc.edu}

\address{Tino Ullrich\\
Hausdorff Center for Mathematics\\ Endenicher Allee 62\\
53115 Bonn, Germany} \email{tino.ullrich@hcm.uni-bonn.de}
\begin{abstract} For $1<p<\infty$ we determine the precise range
of $L_p$ Sobolev spaces for which the Haar system is an unconditional basis.
We also consider the natural  extensions to Triebel-Lizorkin spaces and prove upper and lower
bounds for norms of projection operators depending
on properties of the  Haar frequency set.
\end{abstract}

\date\today

\thanks{Research supported in part
by National Science Foundation grants DMS 1200261 and DMS 1500162.
A.S.  thanks the Hausdorff Research Institute for Mathematics in Bonn for
support. The paper was initiated in the summer of 2014 when he participated in
the Hausdorff Trimester Program in Harmonic Analysis and Partial Differential
Equations. Both authors would like to thank Winfried Sickel and Hans Triebel
for several valuable remarks.}
\maketitle



\section{Introduction}
\noindent We consider the Haar system on the real line given by
\Be\label{HaarS}
  \sH = \{h_{j,\mu}~:~\mu\in \bbZ, j=-1,0,1,2,...\}\,,
\Ee
where for $j \in \bbN\cup \{0\}$, $\mu\in \bbZ$, the function $h_{j,\mu}$ is
defined by $$h_{j,\mu}(x)=\bbone_{I^+_{j,\mu}}(x)-\bbone_{I^-_{j,\mu}}(x)\,,$$
and $h_{-1,\mu}$ is the characteristic function of the interval
$[\mu,\mu+1)$. The intervals
$I^{+}_{j,\mu}=[2^{-j}\mu, 2^{-j}(\mu+1/2))$ and
$I^{-}_{j,\mu}=[2^{-j}(\mu+1/2), 2^{-j}(\mu+1))$ represent the dyadic children
of the usual dyadic interval $I_{j,\mu}=[2^{-j}\mu, 2^{-j}(\mu+1))$.

It has been shown by Marcinkiewicz  \cite{marcinkiewicz}
(based on Paley's square function result  \cite{paley} for the Walsh system) that the Haar system, in contrast to the trigonometric
system, represents an unconditional basis in all $L_p([0,1])$ if $1<p<\infty$.
In this paper we consider this problem in Banach spaces measuring
smoothness. Triebel \cite{triebel73, triebel78,
triebel-bases} showed that the Haar system represents an unconditional basis in
Besov spaces $B^s_{p,q}$ if $1<p,q<\infty$ and $-1/p' < s < 1/p$. In
addition, he obtained extensions to quasi-Banach spaces. See also Ropela
\cite{ropela}, Sickel \cite{sickel}, and Bourdaud \cite{bourdaud} for related
results. Note, that the
endpoint case $s=1/p$ (and by duality the case $s=-1/p'$) can be excluded by
noting that all Haar functions belong to $B^{1/p}_{p,q}$ if and only if
$q=\infty$. Concerning Sobolev and Triebel-Lizorkin spaces the picture is far
more interesting. Triebel \cite{triebel-bases} proved that the Haar system is an
unconditional basis in Sobolev
(or Bessel potential) spaces $L_p^s$, $1<p<\infty$,  if $\max\{-1/p',1/2\}<s<\min\{1/p,1/2\}$. Here the norm in $L_p^s$ is given by  $\|f\|_{L_p^s}=
\|\cD^s_B f\|_p$ where  $\widehat {\cD_B^sf}(\xi)=(1+|\xi|^2)^{s/2}\widehat
f(\xi)$.

It has been an open question (formulated explicitly by Triebel in \cite[p.95]{triebel-bases} and
in \cite{triebelproblem})
whether the Haar system  is an
unconditional Schauder basis on   $L_p^s$ for
 the ranges $1<p<2$, $1/2\le s\le 1/p$ and
 $2<p<\infty$, $-1/p'\le s\le - 1/2$.
We answer this question negatively.

It is natural to formulate the results
in  the class of  Triebel-Lizorkin spaces $F^s_{p,q}$ which include the
$L_p$-Sobolev spaces $L^s_p$; recall that by Littlewood-Paley theory $L_p^s=
F^s_{p,2}$
for $1<p<\infty$ and $s\in \bbR$.  We emphasize that the results are already
new  for the special case of $L_p^s$-spaces.
\begin{theorem} \label{Fthm} For  $1<p,q<\infty$, the Haar system is an
unconditional basis in $F^{s}_{p,q}$ if and only if
$$\max\{-1/p',-1/q'\} <s< \min\{1/p,   1/q\}\,.$$
\end{theorem}

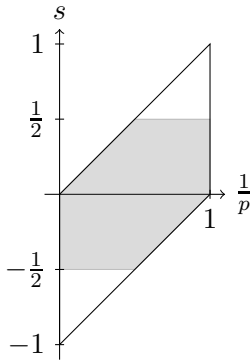
\begin{figure}
 \begin{center}
\begin{tikzpicture}[scale=2]
\draw[->] (-0.1,0.0) -- (1.1,0.0) node[right] {$\frac{1}{p}$};
\draw[->] (0.0,-1.1) -- (0.0,1.1) node[above] {$s$};

\draw (1.0,0.03) -- (1.0,-0.03) node [below] {$1$};
\draw (0.03,1.0) -- (-0.03,1.00) node [left] {$1$};
\draw (0.03,.5) -- (-0.03,.5) node [left] {$\frac{1}{2}$};
\draw (0.03,-.5) -- (-0.03,-.5) node [left] {$-\frac{1}{2}$};
\draw (0.03,-1.0) -- (-0.03,-1.00) node [left] {$-1$};

\draw[fill=black!70, opacity=0.2] (0.0,-.5) -- (0.0,0.0) -- (.5,.5) -- (1.0,0.5) -- (1.0,0.0) -- (.5,-.5) -- (0.0,-.5);
\draw (0.0,-1.0) -- (0.0,0.0) -- (1.0,1.0) -- (1.0,0.0) -- (0.0,-1.0);
\draw (0.85,0.65) node {};
\draw (0.15,-0.65) node {};
\end{tikzpicture}
\caption{Domain for an unconditional basis in spaces $L^s_p$}\label{fig1}
\end{center}
\end{figure}
Thus  the result about the  Haar system in Sobolev and
 Triebel-Lizorkin spaces depends in a  significant way on the secondary
integrability parameter $q$ while for the Besov spaces $q$ plays no role.
The ``if'' part of Theorem \ref{Fthm} was known and can be found in
\cite{triebel-bases}.
The figure above illustrates the differences of the results in  Besov and
Sobolev spaces.

An application of Theorem \ref{Fthm} concerns dyadic characterizations of $F^s_{p,q}$.
For $j\ge 1$ let $\bbone_{j,\mu}$ be the characteristic function of the support
of $h_{j,\mu}$.
One defines the sequence space $f^{s}_{p,q}$ as the space of all doubly-indexed
sequences $\{\lambda_{j,\mu}\}_{j,\mu} \subset \bbC$ for which
\begin{equation}\label{ss}
\|f\|_{f^{s}_{p,q}}=
\Big\|\Big(\sum\limits_{j=-1}^\infty2^{jsq}\Big|\sum\limits_{\mu\in \bbZ}
\lambda_{j,\mu}\bbone_{j,\mu}\Big|^q\Big)^{1/q}\Big\|_p
\end{equation}
is finite. For $f \in \cS(\bbR)$ consider the dyadic version
of the $F^{s}_{p,q}$-norm given by
$$
    \|f\|_{F^{s,{\rm dyad}}_{p,q}} = \|\{2^j\langle f,h_{j,\mu}\rangle
\}_{j,\mu}\|_{f^s_{p,q}}
$$
and let $F^{s,{\rm dyad}}_{p,q}$ be the completion of $\cS(\bbR)$ under this
norm.

Triebel \cite{triebel-bases} showed that $\max\{-1/p',-1/q'\} < s <
\min\{1/p,1/q\}$ is sufficient for $F^{s,{\rm dyad}}_{p,q} = F^s_{p,q}$ with
equivalence of norms. He also showed that this
equivalence implies that $\sH$ is an unconditional basis in
$F^s_{p,q}$. Hence, Theorem \ref{Fthm} yields the necessity of Triebel's
result:
\begin{corollary}\label{waveiso} For  $1<p,q<\infty$ we have  $F^{s,{\rm
dyad}}_{p,q} = F^s_{p,q}$ if and only if
$$\max\{-1/p',-1/q'\} <s< \min\{1/p,   1/q\}\,.$$
\end{corollary}

\subsection*{\it Quantitative results}
We now formulate quantitative versions of  Theorem \ref{Fthm}.
For $j\ge 0$ define  the {\it Haar frequency} of $h_{j,\mu}$ to be  $2^j$. For
any subset $E$ of the Haar system let  $\text{HF}(E)$ be the {\it Haar frequency
set} of $E$, i.e.  $\text{HF}(E)$ consists of all $2^k$ with $k\ge 0$ for which
there exists  $\mu\in \bbZ$ with $h_{k,\mu}\in E$.
Let $P_E$ be the orthogonal projection to the subspace spanned by $\{h:h\in
E\}$, which is closed in $L_2(\bbR)$. For Schwartz functions $f$ we define
$$P_E f= \sum_{h_{j,\mu}\in E} 2^j\inn {f}{h_{j,\mu}}h_{j,\mu}.$$
For function spaces $X$ such as $X=F_{p,q}^s$ (or $X=B_{p,q}^s$)
we define   growth functions  depending on the cardinality
of the Haar frequency set of $E$.
First, for any $A\subset \{2^n: n=0,1,\dots\}$, set
\Be\label{growthA}
\cG(X,A)= \sup \big\{\|P_E\|_{X\to X}: \text{HF}(E)\subset A\big\}.
\Ee
Define, for $\La\in \bbN$,  the upper and lower {\it Haar projection numbers}
\Be\label{gastar}
\begin{aligned}
\ga^*(X;\La) &= \sup \big\{\cG(X,A): \# A\le \Lambda\}\,,
\\
\ga_*(X;\La) &= \inf  \big\{\cG(X,A): \# A\ge \La\}\,.
\end{aligned}
\Ee

Clearly, $\gamma_*(X;\La)\le \gamma^*(X;\La)$. If the Haar basis is an unconditional basis of $X$ then
$\gamma^*(X;\La)= O(1)$. By the known results we have
$\gamma^*(F^s_{p,q};\La)=O(1)$ for the cases
$\max\{-1/p', -1/q'\} <s< \min\{1/p, 1/q\}\,.$
Note that for
$s\ge 1/p$  the Haar functions do not belong to $F^s_{p,q}$, and thus
$\gamma_*(F^s_{p,q};\La)=\infty$. By duality,
$\gamma_*(F^s_{p,q};\La)=\infty$ for $s\le -1+1/p$.
Unlike for the scale of Besov spaces, there are intermediate ranges where
the Haar system is not an unconditional basis of $F^s_{p,q}$  but the Haar
projection numbers are finite, however not uniformly bounded.

In what follows we always assume $\La>10$. We shall use the notation
$A\lc B$, or $B\gc A$,  if $A\le CB$ for a positive constant depending only on
 $p,q,s$. We also use $A\approx B$ if both $A\lc B$ and $B\lc A$.

\begin{theorem} \label{sharpGabd}
(i) For $1<p<q< \infty$, $1/q< s< 1/p$,
$$\gamma_*(F^s_{p,q};\La) \approx \gamma^*(F^s_{p,q};\La)
\approx
\La^{s-\frac 1q}\,.
$$

(ii)  For $1<q<p<\infty$, $-1/p' < s<-1/q'$,
$$
\gamma_*(F^s_{p,q};\La) \approx \gamma^*(F^s_{p,q};\La)
\approx \La^{-\frac1{q'}-s}\,.
$$
\end{theorem}

Consequently the magnitude of  $\cG(F^s_{p,q},A)$ depends on the
cardinality of $A$ alone and we have $\cG(F^s_{p,q},A)\approx
(\#A)^{s-1/q}$ when $1/q<s<1/p$.
For the endpoint case $s=1/q$ or $s=-1/q'$ we still have failure of unconditional convergence,  but a new phenomenon occurs:
the quantity $\cG(F^{1/q}_{p,q},A) $ also depends on the density of $\log_2(A)=\{k: 2^k\in A\}$ on intervals of length $\approx \log_2\!\#A$. Define for any $A$ with $\#A\ge 2$
\begin{align*}
\overline {\cZ}(A)&= \max_{n\in \bbZ}
\# \{k: 2^k\in A, \, |k-n|\le\log_2\!\#A \}\,,
\\
\underline{\cZ}(A)&=\min_{2^n\in A}
\# \{k: 2^k\in A, \, |k-n|\le \log_2\#A \}\,.
\end{align*}
Notice that
$1\le \underline{\cZ}(A)\le \overline{\cZ}(A) \le1+ 2\log_2\!\#A$.

\begin{theorem} \label{GXAbds}
Let $A\subset \{2^n: n\ge 0\}$ such that $\#A\ge 2$.

(i) For  $1<p<q< \infty$,
$$
\underline{\cZ}(A)^{ 1-\frac 1q}\lc
\frac{ \cG(F^{1/q}_{p,q},A)}
{(\log_2\#A)^{\frac 1q} }
\lc
\overline{\cZ}(A)^{ 1-\frac1q}\,.
$$

 (ii) For $1<q<p<\infty$,
 $$
 \underline{\cZ}(A)^{\frac 1q}
 \lc \frac{\cG(F^{-1/q'}_{p,q},A)}{(\log_2\#A)^{1-\frac 1q} }\lc  \overline{\cZ}(A)^{\frac 1q}\,.
$$
\end{theorem}

We remark  that $\overline{\cZ}(A)=O(1)$ when  $\# A\approx 2^N$ and $\log_2(A) $ is $N$-separated.  On the other hand, for $A=[1,2^N]\cap \bbN$ we have  $\underline{\cZ}(A)\ge N$.
Hence it follows that
the lower and upper Haar projection numbers for the endpoint   cases
 have now different growth rates:

\begin{corollary} \label{endpointGabds}
For $\Lambda\ge 4$ we have the following equivalences.

(i) For $1<p<q< \infty$,
$$\gamma_*(F^{1/q}_{p,q};\La)
\approx
(\log_2 \La)^{1/q}
$$
and
$$\gamma^*(F^{1/q}_{p,q};\La)
\approx
\log_2 \La\,.
$$

(ii)  For $1<q<p<\infty$,
$$
\gamma_*(F^{-1+1/q}_{p,q};\La)
\approx (\log_2 \Lambda)^{1-1/q}
$$
and
$$
\gamma^*(F^{-1+1/q}_{p,q};\La)
\approx \log_2 \Lambda\,.
$$
\end{corollary}

The proof of the lower bounds for the lower Haar projection numbers also shows
that
for any infinite subset  $A$
 of $\{2^n:n\ge 0\}$ there is a subset $E$ of the Haar system, with Haar frequency set contained in $A$,  so that  $P_E$ does not extend to a bounded operator
on $F^s_{p,q}$,
in the $s$-ranges of Theorem \ref{sharpGabd}.

\subsection*{\it Guide through the paper} In \S\ref{prelsect} we discuss some preliminary facts about  Peetre maximal functions and Triebel-Lizorkin spaces.
In \S\ref{upperbdsect} we prove the sharp upper bounds for Haar projection operators. In \S\ref{testfctsect} we provide estimates for suitable families of test functions in $F^s_{p,q}$ for $p>q$ and $s\le -1/q'$.
In \S\ref{lowerbdsect-mosts}  we determine the behavior of
$\gamma_*(F^s_{p,q};\La)$ for large  $\La$, and also the
behavior of
$\gamma^*(F^s_{p,q};\La)$ if $s<-1/q'$.
 In
 \S\ref{gammaupperstar}  we prove refined  lower bounds for the endpoint
$s=-1/q'$, $p\ge q$ which  yield in particular  precise bounds for
$\gamma^*(F^{-1/q'}_{p,q};\La)$. Concluding remarks are made in \S\ref{remarksect}.

\section{Preliminaries}
\label{prelsect}

\subsection{Littlewood-Paley decompositions and Triebel-Lizorkin spaces}
\label{LPsect}
We pick  functions $\psi_0$, $\psi$ such that $|\widehat \psi_0(\xi)|>0$ on
$(-\eps,\eps)$ and $|\widehat \psi(\xi)|>0$ on $\{\xi:\eps/4 < |\xi| < \eps)$
for some fixed $\eps>0$.
We further assume
\Be\label{psicanccond} \int\psi(x) x^n dx =0 \text{ for } n=0,1,\dots, M_1\Ee
(if $M_1$ is a large given integer).

Let now $\varphi_0 \in \cS(\bbR)$ be a compactly supported function with
$\varphi_0 \equiv 1$ on $[-4/3,4/3]$ and $\varphi_0 \equiv 0$ on $\bbR\setminus
[-3/2,3/2]$\,. Putting $\varphi = \varphi_0 - \varphi_0(2\cdot)$ we obtain a
smooth dyadic decomposition of unity, i.e.,
$\varphi_0(\cdot)+\sum_{k\geq 1} \varphi(2^{-k}\cdot)\equiv 1$.
In addition, we set $\beta_0(\xi) := \varphi_0(2\xi/\eps)/\psi_0(\xi)$ and
$\beta(\xi) := \varphi(2\xi/\eps)/\psi(\xi)$. Hence, $\beta_0, \beta$ are
well-defined Schwartz functions supported on $(-3\eps/4,3\eps/4)$ and
$\{\xi:\eps/3<|\xi|<3\eps/4\}$, respectively, such that
\Be\label{Calderon}
\widehat \psi_0(\xi)  \beta_0(\xi) +\sum_{k=1}^\infty
\widehat \psi(2^{-k}\xi) \beta(2^{-k}\xi)=1 \text{ for all $\xi\in \bbR$\,,}
\Ee
see also \cite{fr-ja, FrJa90}. The Triebel-Lizorkin space $F^s_{p,q}(\bbR)$
is usually defined via a smooth dyadic decomposition of unity on the Fourier
side, generated for instance by $\varphi_0$ and $\varphi$ defined above.
We define
the operators $L_k$ by
$\widehat {L_0f}(\xi)=\varphi_0(\xi) \widehat f(\xi)$ and
$\widehat {L_kf}(\xi)=\varphi(2^{-k}\xi) \widehat f(\xi)$ and obtain the usual
example for an inhomogeneous Littlewood-Paley decomposition. In particular
\Be \label{repr}
f=\sum_{k=0}^\infty L_k f
\Ee holds for all Schwartz functions $f$, with convergence in $\cS'(\bbR)$ and
all
$L_p(\bbR)$. For $0<p<\infty$, $0<q\leq \infty$ and $s\in \bbR$ the
Triebel-Lizorkin space $F^s_{p,q}(\bbR)$ is the collection of all tempered
distributions $f\in \cS'(\bbR)$ such that
\Be\label{TL}
  \|f\|_{F_{p,q}^s}= \Big\|\Big(\sum_k2^{ksq}| L_k
f|^q\Big)^{1/q}\Big\|_p
\Ee
is finite (usual modification in case $q=\infty$).

Based on \eqref{Calderon} it can be proved using
vector-valued singular integrals \cite{BCP}, see also
\cite[\S 2.4.6.]{triebel2} and \cite{Ry99}, that
\Be\label{localmeans}
\|f\|_{F^s_{p,q}} \approx \Big\| \Big(\sum_{k=0}^\infty
2^{ksq}|\psi_k*f|^q\Big)^{1/q}
\Big\|_p
\Ee
with $\psi_0,\psi$ from above, $\psi_k(x) = 2^k\psi(2^kx)$ and $M_1+1>s$\,.
First of all this characterization yields a useful version of \eqref{TL} with
operators $\widetilde L_k$ that reproduce the $L_k$. Indeed, it is easy to find
compactly supported Schwartz functions $\widetilde \varphi_0$, $\widetilde
\varphi$ such that $\widetilde \varphi_0(\xi) = 1$ on $\supp\,\varphi_0$,
respectively for $\widetilde \varphi$, and that $\widehat \psi_0 = \widetilde
\varphi_0$ and $\widehat \psi = \widetilde \varphi$ are admissbible for
\eqref{localmeans}. With $\widetilde L_k$ as above we have
$\widetilde L_k L_k=L_k$ for $k=0,1,2,\dots$.

The above characterization \eqref{localmeans} allows for choosing $\psi_0,\psi$
compactly supported.
Characterizations of this type are
termed ``local means'' in  Triebel \cite{triebel2}, \S 2.4.6, and turn out to
be convenient for the purpose of this paper.

\subsection{Peetre maximal functions} \label{Peetremax}

The main tool to estimates operators in $F^s_{p,q}$ spaces are the
vector-valued maximal inequalities by Fefferman-Stein \cite{feffstein}  and a
variant due to Peetre \cite{Pe}.
We shall need a (variant of) an endpoint version for the  Peetre maximal operators which was proved
in \S6.1 of \cite{cs-lms} using an argument involving the
$\#$-function  of Fefferman-Stein \cite{feffstein-Hp}. Let $\cE(r)$ be the
space of all tempered distributions whose Fourier transform is supported in $
\{\xi:|\xi|\le r\}$.
Let
\Be \label{peetremod}\fM^{r}_n g(x)= \sup_{|y|\le 2^{n+2}/r}|g(x+y)|.\Ee
Then, the one-dimensional version of the  result in \cite{cs-lms} states that
for any sequence of positive numbers $r_k$,  $0<p,q<\infty$,
and for any sequence of functions $f_k \in \cE(r_k)$,
\Be \label{peetremodineq}
\Big\|\Big(\sum_{k} |\fM_n^{r_k} f_k|^q\Big)^{1/q}\Big\|_p
\lc \max \{ 2^{n/p}, 2^{n/q}\}
\Big\|\Big(\sum_{k} | f_k|^q\Big)^{1/q}\Big\|_{p}\,.
\Ee
The original result by Peetre is equivalent with the similar inequality with constant
$C_\eps 2^{n\eps} \max \{ 2^{n/p}, 2^{n/q}\}$ on the right hand side. Here we
also need a vector-valued version with variable $n$. We formulate it for $p\ge
q$ since this is the version used here.

\begin{proposition} \label{vectpeetremodprop}
Let $0<q\le p<\infty$. For any sequence of positive numbers $r_k$,
and for any doubly indexed sequence of functions $\{f_{k,n}\}_{k,n\ge 0}$, with  $f_{k,n} \in \cE(r_k)$,
\Be \label{vectpeetremodineq}
\Big\|\Big(\sum_{k,n} 2^{-n}|\fM_n^{r_k} f_{k,n}|^q\Big)^{1/q}\Big\|_p
\lc
\Big\|\Big(\sum_{k,n} | f_{k,n}|^q\Big)^{1/q}\Big\|_p\,.
\Ee
\end{proposition}
We omit the proof since it is a straightforward variant of
the argument in  \S6.1 of \cite{cs-lms}.


\subsection{Duality}\label{dualitysect}
We show that the
statements for (i) and (ii) in Theorems \ref{sharpGabd} and \ref{endpointGabds} are equivalent,  by duality.

Given an integral operator $T$ acting on Schwartz functions
let  $T'$ denote the transposed operator with the property
$$\int Tf(x) g(x) dx= \int f(y) T' g(y) dy.$$ Note that for the Haar projection
operators $P_E$ we have  $P_E=P_E'$.  Also if $Tf=K*f$, the operator of
convolution then $T'$ is the operator of convolution with $K(-\cdot)$.

Assume that $1<p,q<\infty$ and let $s\in \bbR$ such that $P_E: F^{-s}_{p',q'}
\to F^{-s}_{p',q'}$ is bounded  with operator norm $A$.
 Then we need to  show that
\Be \label{dual}\Big\|\Big(\sum_k 2^{k sq}| L_k  P_E f|^{q}\Big)^{1/q}\Big\|_p\lc
A\Big\|\Big(\sum_k2^{ksq}| L_k f|^q\Big)^{1/q}\Big\|_p\,,
\Ee
with implicit constant depending only on $p,q,s$ and the choice of the Schwartz functions defining $L_k$.

To see \eqref{dual} we may assume that the $k$-sum on the left hand side is
extended over a finite subset $\fK$ of $\bbN\cup\{0\}$.
Then  there  is $G=\{G_k\}\in L_{p'}(\ell_{q'})$ with
$\|G\|_{L_{p'}(\ell_{q'})}\leq 1$ so that the left hand side of \eqref{dual} is
finite and equal to
\begin{align*}
&\int \sum_{k=0}^\infty 2^{ks} L_k P_Ef (x)
G_k(x) dx\,
\notag \\
&=
\sum_{k=0}^\infty 2^{ks}  \sum_{j=0}^\infty
\int L_k P_E
  L_j \widetilde L_jf (x) G_k(x) dx
\\&=
\int \sum_{j=0}^\infty \widetilde L_j f (y) L_j'P_E'\big[{\scriptstyle\sum}_k 2^{ks}
L_k'G_k\big](y) dy.
\end{align*}
Using
H\"older's inequality, we estimate the last displayed expression
by
\begin{align*}&\Big\|\Big(\sum_j2^{jsq}|\widetilde L_j f|^q\Big)^{1/q} \Big\|_p
\,\Big\| \Big(\sum_{j=0}^\infty 2^{-jsq}|
 L_j'P_E\big[{\scriptstyle\sum}_k
2^{ks}L_k'G_k \big]|^{q'} \Big)^{1/q'}\Big\|_{p'}
\\
&\lc \|f\|_{F^{s}_{p,q}} \|P_E\big[{\scriptstyle\sum}_k  2^{ks} L_k' G_k\big]\|_{F^{-s}_{p',q'}}
\,\lc \|f\|_{F^{s}_{p,q}} A
\Big\|\sum_k 2^{ks} L_k' G_k\Big\|_{F^{-s}_{p',q'}}
\end{align*}
by assumption.
Finally
\begin{align*}
\Big\|\sum_k 2^{ks} L_k' G_k\Big\|_{F^{-s}_{p',q'}}&=
\Big\|\Big(\sum_{j}2^{-jsq}\Big|\sum_{k=j-2}^{j+2} 2^{ks}L_j L_k' G_k
\Big|^{q'}\Big)^{1/q'}\Big\|_{p'}
\\
&\le C_{p'} \Big\|\Big(\sum_k |G_k|^{q'}\Big)^{1/q'}\Big\|_{p'}  \lc 1\,,
\end{align*}
and \eqref{dual} is proved.


\section{Upper bounds for   Haar  projections}
\label{upperbdsect}
For the upper bounds asserted in Theorem \ref{sharpGabd} it suffices to consider  the projection numbers
$\gamma^*(F^s_{p,q};\Lambda)$ and $\gamma_*(F^s_{p,q};\Lambda)$
for the choice of  $\La={2^N}$, for large $N$.
The following theorem gives a refined version of these upper bounds.

For  a subset $E$ of the Haar system let
\Be\label{cZ} Z_N(E)=\max_{k\in \bbN} \#\{n: 2^n\in \HF(E),\,|n-k|\le N\}
\Ee
Clearly $1\le Z_N(E)\le 2N+1$ for all $E\subset \sH$.

\begin{theorem} \label{upperlabounds}
Let $1<q<p<\infty$, $N\ge 2$. There is $N_0(p,q,s)$ such that for $N\ge
N_0(p,q,s)$ the following holds for subsets $E$ of the Haar system with
$\#\text{\rm HF}(E)\le2^{N}$ (with implicit constants depending on $p,q,s$).

(i)  If $-1/p'<s<-1/q'$ then
$$\|P_E\|_{F^s_{p,q}\to F^s_{p,q}} \lc 2^{N(-s-\frac 1{q'})}.$$

(ii) For the case  $s=-1/q' $ we have
$$\|P_E\|_{F^{-1/q'}_{p,q}\to F^{-1/q'}_{p,q}} \lc N^{1-1/q} Z_N(E)^{1/q}\,.
$$
\end{theorem}

The remainder of this section is devoted to the proof of Theorem \ref{upperlabounds}.
\subsection*{\it Two preliminary estimates}

We state two lemmata which will be used frequently when estimating the Haar
projection operators $P_E$. In what follows let $\psi_k$ be as in \S
\ref{LPsect}.

\begin{lemma} \label{psicanc}
Let  $k\le j$. Then, with  $y_{j,\mu}:= 2^{-j}(\mu+\tfrac 12)$, the support of
$h_{j,\mu}*\psi_k$ is contained in $[y_{j,\mu} - 2^{-k},y_{j,\mu} +2^{-k}].$ Moreover,
$$\|h_{j,\mu}*\psi_k\|_\infty \lc 2^{2k-2j}.$$
\end{lemma}
\begin{proof}
The support property is immediate due to the support property of $\psi_k$. Since $\int h_{j,\mu}(y) dy=0$ we have
$$h_{j,\mu}*\psi_k(x)
= \int\big(\psi_k(x-y)-\psi_k(x-y_{j,\mu})\big)h_{j,\mu}(y) dy$$
and using  $\psi_k'=O(2^{2k})$, we get
\[|h_{j,\mu}(x)|\lc \int_{y_{j,\mu}-2^{-j}}^{y_{j,\mu}+2^{-j}} 2^{2k} |y-y_{j,\mu}|dy\lc 2^{2k-2j}\,.\qedhere\]
\end{proof}

\begin{lemma} Let $0< p \leq \infty$.\label{canc-const}
(i) Suppose that $k\ge j$, and let $x\in \bbR$ such that
$$ \min\,\big \{ |x-2^{-j}\mu|,\, |x-2^{-j}(\mu+\tfrac12)|, \,|x-2^{-j}(\mu+1)|\big\} \,\ge 2^{-k}.$$
Then $h_{j,\mu}*\psi_k(x)= 0$.

(ii)
$\|h_{j,\mu}*\psi_k\|_p\lc 2^{-k/p}$ for $k\ge j$.
\end{lemma}
\begin{proof}
(i) follows by the support and cancellation properties of $\psi_k$ and
fact that $h_{j,\mu}$ is constant on
$I^+_{j,\mu}$,  $I^-_{j,\mu}$, and  $I^\complement_{j,\mu}$. Since
$\|\psi_k\|_1\lc 1$ we have (ii) for $p=\infty$. From (i) we then get
(ii) for all $p$.
\end{proof}

\subsection*{\it Basic reductions}

We  use the Peetre type maximal operators $\fM_n^r$ defined in
\eqref{peetremod};
it will be convenient to use the notation $\cM^k_n =\fM^{2^{k}}_n$ so that
\Be\label{maxopabbr}
\cM^k_n g(x) =\sup_{|y|\le 2^{-k+n+2}}|g(x+y)|\,.
\Ee


In  the remainder of the chapter we assume  $1<q<\infty$,  and  $E$ will denote  a subset of $\sH$ satisfying
\Be\label{sizeassumption} \#(\text{HF}(E))< 2^{N+1}.\Ee
Let $\psi_k$ be  as in
\S\ref{testfctsect}.   Theorem  \ref{upperlabounds} follows from
 $$
\Big\|\Big(\sum_{k\in \bbN} 2^{ksq}|\psi_k *P_E f|^q\Big)^{1/q}\Big\|_p
\lc \max \{2^{N(\frac 1q-s-1)}, N^{1-1/q}Z^{1/q}\} \|f\|_{F^s_{p,q}},
$$
with $Z:=\cZ_N(E)$.
This in turn  follows from
 \begin{multline} \label{upperLabd}
\Big\|\Big(\sum_{k\in \bbN} 2^{ksq}\big|\psi_k *P_E\big [{\textstyle \sum_l
2^{-ls} \psi_l* f_l }\big]\big|^q\Big)^{1/q}\Big\|_p\\
\lc \max \{2^{N(\frac 1q-s-1)}, N^{1-1/q} Z^{1/q}\} \Big\|
\Big(\sum_{l=0}^\infty |f_l|^q\Big)^{1/q}\Big\|_p
\end{multline}
for all $\{f_l\}$ with $f_l\in \cE(2^l)$.

Given a set $E$ of Haar functions,
we set $E_j=\{\mu: h_{j,\mu}\in E\}$.
We  link $j=k+m$ and $l=k+m+n$ and
define, for $m,n\in \bbZ$, $k=0,1,\dots$,
\begin{multline} \label{Tkmndef}
T^k_{m,n} f = \sum_{\mu\in E_{k+m}}
 2^{k+m}\inn{\psi_{k+m+n}* f}{h_{k+m,\mu}} \psi_k*
h_{k+m,\mu}\,
,\\
\text{ if $k+m\in \HF(E)$, $k+m+n\ge 0$,}
\end{multline}
and $T^k_{m,n} =0$ if  $k+m\notin \HF(E)$ or
$k+m+n< 0$.
Then
\Be \label{Tkmnexpansion}
2^{ks}\, \psi_k *P_E\big [{\textstyle \sum_l
2^{-ls} \psi_l* f_l }\big]= \sum_m\sum_n 2^{-s(m+n)}T^k_{m,n} f_{k+m+n}.
\Ee

In preparation for the proof of \eqref{upperLabd} we first state estimates of $T^k_{m,n} f$ in terms of the Peetre type maximal operators $\cM^k_n$
\eqref{maxopabbr}, or in some cases
just  the Hardy-Littlewood maximal operator $M$.

\begin{lemma} \label{Tkmnptwise}  Let $k\ge 0$, $k+m\ge 0$.
The following estimates hold for continuous $f$.

(i) For $m\ge 0$ and $n\ge 0$,
$$ |T^k_{m,n} f(x)|\lc 2^{-m-n}
\cM_{n+m}^{k+n+m}f(x)\,.
$$

(ii)  For $m\ge 0$ and $n\le 0$,
$$ |T^k_{m,n} f(x)|\lc \begin{cases}  2^{n-m}\cM_{m+n}^{k+m+n}f(x)
\text{ if $n\ge -m$}
\\
2^{n-m} Mf(x) \text{ if $n< -m$}\,.
\end{cases}
$$

(iii) For $m\le 0$ and $n\ge 0$,
$$ |T^k_{m,n} f(x)|\lc
  2^{-n}\cM_{n}^{k+n+m}f(x)\,.
$$

(iv)  For $m\le 0$ and $n\le 0$,
$$ |T^k_{m,n} f(x)|\lc  Mf(x)\,.
$$
\end{lemma}

\begin{proof} Let $m\ge 0$, $n\ge 0$.
 For $x\in \supp (\psi_k* h_{k+m,\mu})$ we have (with $\widetilde \psi:=\psi(-\cdot)$)
\begin{align*}
&2^{k+m}|\inn{\psi_{k+m+n}* f}{h_{k+m,\mu}}|=
2^{k+m}|\inn{f}{\widetilde \psi_{k+m+n}*h_{k+m,\mu}}|
\\&\lc 2^{k+m} 2^{-k-m-n} \sup_{y: |x-y| \le 2^{2-k} }|f(y)|
\\&\lc 2^{-n} \cM_{m+n}^{k+m+n} f(x).
\end{align*}
Now $\psi_k* h_{k+m,\mu}=O(2^{-2m})$ and for fixed $x$ the $\mu$-sum
in \eqref{Tkmndef} contributes $O(2^m)$ terms. This yields (i).

Let  $m\ge 0$, $-m\le n\le 0$. We now have
$\|\widetilde \psi_{k+m+n}*h_{k+m,\mu}\|_\infty\lc 2^{2n}$, by Lemma \ref{psicanc}
and therefore
\begin{align*}
&2^{k+m}|\inn{\psi_{k+m+n}* f}{h_{k+m,\mu}}|=
2^{k+m}|\inn{f}{\widetilde \psi_{k+m+n}*h_{k+m,\mu}}|
\\&\lc 2^{k+m} 2^{2n} 2^{-k-m-n} \sup_{y: |x-y| \le 2^{2-k} }|f(y)|
\\&\lc 2^{n} \cM_{m+n}^{k+m+n} f(x).
\end{align*}
As in the previous case
$\psi_k* h_{k+m,\mu}=O(2^{-2m})$ and there are $O(2^m)$ $\mu$-terms that contribute. This leads to  (ii) in the case when $m+n\ge 0$.
If $m\ge 0$ and  $n\le -m$ then we have instead
\begin{multline*}
2^{k+m}|\inn{f}{\widetilde \psi_{k+m+n}*h_{k+m,\mu}}|
\\
\lc
2^{k+m} 2^{2n}\int_{|x-y|\le 2^{-k-m-n+2}}|f(y)|dy \lc 2^n Mf(x),
\end{multline*}
which gives the second estimate in (ii).

Next assume $m\le 0$. Now we  use that $h_{k+m,\mu}*\widetilde \psi_{k+m+n}$
is supported in the union of three intervals of length $2^{-k-m-n}$ centered at
the endpoints and the middle point of $\supp(h_{k+m,\mu})$.
Thus, for $x\in \supp(\psi_k*h_{k+m,\mu})$,
\begin{multline*}
2^{k+m}|\inn{f}{\widetilde \psi_{k+m+n}*h_{k+m,\mu}}|
\\
\lc
2^{-n} \sup_{|x-y|\le 2^{-k-m+2}}|f(x-y)|
 \lc 2^{-n} \cM_n^{k+m+n}f(x)
\end{multline*}
and (iii) follows since every $x$ is contained in
$\supp(\psi_k*h_{k+m,\mu})$
for at  most three choices of $\mu$.

When $m\le 0$, $n\le 0$ we can estimate instead
$$2^{k+m}|\inn{\psi_{k+m+n}* f}{h_{k+m,\mu}}| \lc Mf(x), $$ for  $x\in \supp(\psi_k*h_{k+m,\mu})$ and obtain (iv).
 \end{proof}

By the Peetre type inequality \eqref{peetremodineq} we see that for $m\ge 0$ the $L_\rho\to L_\rho$ operator norm  of  $T^k_{m,n}$ when acting on functions in $\cE(C2^{k+m+n})$,  is
$O(2^{-(m+n)/\rho'})$ if $n\ge 0$ and
$O(2^{2n} 2^{-(m+n)/\rho'})$ when $n\le 0$. It will be useful to  observe an improvement  in $m$ which we will apply for $\rho=p$ and $\rho=q$.

\begin{lemma}\label{Tmnq} Let $m,k\ge 0$.
Then for $1\le \rho\le \infty$ and $f\in L_\rho$,
$$
\|T^k_{m,n} f\|_\rho \lc \begin{cases}
2^{-n/\rho'}2^{-m} \|f\|_\rho\,, &n\ge 0,
\\
2^{n-m} \|f\|_\rho\,, &n\le 0,
\end{cases}
$$
where the implicit constant depends only on $\rho$.
\end{lemma}

\begin{proof}
We have $\|\psi_k*h_{k+m,\mu}\|_\infty=O(2^{-2m})$,
by  the cancellation property  of $h_{k+m,\mu}$.
We decompose $\bbR$ into dyadic intervals of length $2^{-k}$, labeled $J_{k,\nu}$ for $\nu\in \bbZ$. We say that $\mu\sim \nu$ if $\mu\in E_{k+m}$ and $I_{k+m,\mu}$ intersects $J_{k,\nu}$ or one of its neighbors.

We first examine the case $n\ge 0$. Now
$\psi_{k+n+m}*h_{k+m,\mu}$  is supported on a set $V^{k+m,n}_\mu$ of measure $O(2^{-k-m-n})$,  namely the union of three intervals of length $2^{-k-m-n+1}$ centered at the two endpoints and the midpoint of the interval $I_{k+m,\mu}$.

Thus
\begin{align*}
&\|T^k_{m,n} f\|_\rho\\&\lc \Big(\sum_\nu\int_{J_{k,\nu}}\Big[\sum_{\mu:\mu\sim\nu}
|\psi_k*h_{k+m,\mu}(x)|2^{k+m}
|\inn{h_{k+m,\mu}*\psi_{k+m+n}}{f}|\Big]^\rho dx\Big)^{1/\rho}
\\
&\lc 2^{-k/\rho} 2^{-2m} 2^{k+m}
\Big(\sum_\nu\Big[\sum_{\mu:\mu\sim\nu}\int_{V^{k+m,n}_\mu}|f(y)|dy \Big]^\rho\Big)^{1/\rho}
\\
&\lc 2^{k/\rho'}2^{-m}
\Big(\sum_\nu\sum_{\mu:\mu\sim\nu}\int_{V^{k+m,n}_\mu}|f(y)|^\rho dy\, 2^{-(k+m+n)\rho/\rho'}
2^{m\rho/\rho'}
\Big)^{1/\rho}
\end{align*}
since the measure of  $\meas(V^{k+m,n}_\mu)=O(2^{-k-m-n})$ and since
for each $\nu$ there are at most $O(2^m)$ integers $\mu$ with $\mu\sim \nu$.
Now each $y$ is contained in a bounded number of the sets
$V^{k+m,n}_\mu$ and for each $\mu$ the number of $\nu$ with $\mu\sim\nu$ is
also bounded. Hence the expression on the last displayed line is dominated by a
constant times
\begin{equation*}
2^{-m}  2^{-n/\rho'}
\Big(\sum_\nu\sum_{\mu:\mu\sim\nu}\int_{V^{k+m,n}_\mu}|f(y)|^\rho dy \Big)^{1/\rho}
\lc 2^{-m}  2^{-n/\rho'} \|f\|_\rho\,.
\end{equation*}
which proves the assertion for $n\ge 0$.

For the case $n\le 0$ we
have  $\|\psi_{k+m+n}*h_{k+m,\mu}\|_\infty=O(2^{2n})$.
The function is supported in an interval $I_{k+m,\mu}^n$ of length
$C2^{-k-m-n}$, centered at an $x_{k+m,\mu}$ in the support of $h_{k+m,\mu}$.
We  estimate
\begin{align*}
& \Big(\sum_\nu\int_{J_{k,\nu}}\Big[\sum_{\mu:\mu\sim\nu}
|\psi_k*h_{k+m,\mu}(x)|2^{k+m}
|\inn{h_{k+m,\mu}*\psi_{k+m+n}}{f}|\Big]^\rho dx\Big)^{1/\rho}
\\
&\lc 2^{-k/\rho} 2^{-2m} 2^{k+m}2^{2n}
\Big(\sum_\nu\Big[\sum_{\mu:\mu\sim\nu}\int_{I_{k+m,\mu}^n}|f(y)|dy \Big]^\rho\Big)^{1/\rho}
\\
&\lc 2^{k/\rho'} 2^{-m} 2^{2n}
\Big(\sum_\nu2^{m\rho/\rho'}\sum_{\mu:\mu\sim\nu}
2^{-(k+m+n)\rho/\rho'}\int_{I_{k+m,\mu}^n}|f(y)|^\rho dy \Big)^{1/\rho}.
\end{align*}
We now interchange summations and integration and  observe that for each $\mu$
there are only $O(1)$ values of $\mu$ with $\mu\sim\nu$.
 This leads to
\begin{align*}\|T^k_{m,n} f\|_\rho\lc2^{-m}2^{2n-n/\rho'}
\Big(\int|f(y)|^\rho\,\#\{\mu:y\in I^n_{k+m,\nu}\} dy\Big)^{1/\rho}
\end{align*}
and since  for each  $y$ there are
at most  $O(2^{-n})$ values of $\mu$ with
$y\in I^n_{k+m,\mu}$ the asserted inequality for $n\le 0$ follows.
\end{proof}

In what follows we use operators $U_k$ defined by $\widehat {U_kf }(\xi)=
\Phi (2^{-k} \xi)\widehat f(\xi) $ where $\Phi\in C^\infty_c(\bbR)$ supported in
$(-4,4)$ satisfying  $\Phi(\xi)=1 $ for $|\xi|\le 2$.  Notice that $U_k
f_k=f_k$ for $f_k\in \cE(2^k)$. In order to facilitate interpolation we
shall replace $f_{k+m+n}$ on the right hand side by $U_{k+m+n} g_{k+m+n}$
where
$\vec g=\{g_l\}_{l=0}^\infty$ is an arbitrary function in $L_p(\ell_q)$.

The main inequalities needed to prove
Theorem  \ref{upperlabounds} for the case $s< -1/q'$
are  stated in the following proposition
(which also provides useful information for the case $s=-1/q'$).
Recall that $\#\HF(E)\le 2^{N+1}$.
\begin{proposition} \label{Tkmn-seq}
 Let $1<q<p<\infty$,  and  $\eps>0$.

(i)  For $m\ge 0$ and $n\ge 0$,
\begin{multline*}
\Big\|\Big(\sum_k |
T^k_{m,n}  U_{k+m+n}g_{k+n+m}|^q\Big)^{1/q} \Big\|_p
\\ \lc_\eps
 \min\{  2^{-\frac{n}{q'}}  2^{-m(1-\frac{1}{q}+\frac 1p+\eps)} ,\,
2^{N(\frac 1q-\frac 1p)}2^{-\frac n{p'}}2^{-m}   \} \|\vec g\|_{L_p(\ell_q)}\,.
\end{multline*}

(ii)  For $m\ge 0$ and $-m\le n\le 0$,
\begin{multline*}
\Big\|\Big(\sum_k |T^k_{m,n} U_{k+m+n}g_{k+m+n}|^q\Big)^{1/q} \Big\|_p\\
\lc 2^{n-m} \min\{ 2^{(m+n)(\frac 1q-\frac 1p+\eps)}, 2^{N(\frac 1q-\frac 1p)}\}
 \|\vec g\|_{L_p(\ell_q)}\,.
\end{multline*}

(iii) For $m\ge 0$ and $n\le -m\le 0$,
\[
\Big\|\Big(\sum_k |T^k_{m,n} U_{k+m+n}g_{k+n+m}|^q\Big)^{1/q} \Big\|_p
\lc 2^{n-m}
 \|\vec g\|_{L_p(\ell_q)}\,.
\]

(iv) For $m\le 0$ and $n\ge 0$,
\begin{multline*}
\Big\|\Big(\sum_k |T^k_{m,n} U_{k+m+n}g_{k+m+n}|^q\Big)^{1/q} \Big\|_p\\
\lc  \min\{ 2^{-n (1-\frac 1q)}, 2^{N(\frac 1q-\frac 1p)}2^{-n(1-\frac 1p)}\}
\|\vec g\|_{L_p(\ell_q)}\,.
\end{multline*}

(v) For $m\le 0$ and $n\le 0$,
\[
\Big\|\Big(\sum_k |T^k_{m,n} U_{k+m+n}g_{k+m+n}|^q\Big)^{1/q} \Big\|_p
\lc  \|\vec g\|_{L_p(\ell_q)}\,.
\]
\end{proposition}


\begin{proof}
We prove (i) by interpolation.
We first observe that, by Lemma \ref{Tmnq},
\[
\Big\|\Big(\sum_k |T^k_{m,n} U_{k+m+n}g_{k+m+n}|^q\Big)^{1/q} \Big\|_q
 \lc 2^{-m} 2^{-n/q'}  \|\vec g\|_{L_q(\ell_q)}\,.
\]
By Lemma \ref{Tkmnptwise}, (i), and \eqref{peetremodineq},
\begin{align*}
&\Big\|\Big(\sum_k |T^k_{m,n} U_{k+m+n}g_{k+m+n}|^q\Big)^{1/q} \Big\|_r
\\
& \lc 2^{-m/q'} 2^{-n/q'} \Big\|\Big(\sum_k |U_{k+m+n}g_{k+m+n}|^q\Big)^{1/q} \Big\|_r
\notag
\\& \lc 2^{-m/q'} 2^{-n/q'} \|\vec g\|_{L_r(\ell_q)}
 \label{Tkmn-rlarge}
\end{align*}
which we choose for $r\gg p>q$ large.
Interpolation  yields
\Be\label{interpol}
\Big\|\Big(\sum_k |T^k_{m,n} U_{k+m+n}g_{k+m+n}|^q\Big)^{1/q} \Big\|_p
 \lc_\eps 2^{-m(1-\frac 1q+\frac 1p-\eps)} 2^{-n/q'}
 \|\vec g\|_{L_p(\ell_q)}
\Ee
with $\eps=\frac{1-q/p}{r-q}$.
Letting $r\to \infty$ we obtain the first bound stated in (i).
We also have by H\"older's inequality ($q<p$)  and $\#\HF(E)\le 2^N$
\begin{align}
&\Big\|\Big(\sum_k |T^k_{m,n} U_{k+m+n}g_{k+m+n}|^q\Big)^{1/q} \Big\|_p
\notag\\
&\lc 2^{N(\frac 1q-\frac 1p)}
\Big(\sum_k \|T^k_{m,n} U_{k+m+n}g_{k+m+n}\|_p^p\Big)^{1/p}
\notag
\\
&\lc 2^{N(\frac 1q-\frac 1p)}
  2^{-m} 2^{-n/p'} \|\vec g\|_{L_p(\ell_p)}
\notag
\\
\label{HoelderboundN}&\lc 2^{N(\frac 1q-\frac 1p)}
  2^{-m} 2^{-n/p'} \|\vec g\|_{L_p(\ell_q)}\,.
\end{align}
Here, for the second inequality we have used Lemma \ref{Tmnq}, with $\rho=p$
and for the third the embedding $\ell_q\subset \ell_p$.
This concludes the proof of (i).

Inequalities (ii), (iii), (iv) follow more directly from the corresponding
statements in Lemma \ref{Tkmnptwise}, combined with an application of
\eqref{peetremodineq}.
\end{proof}

For the case $s=-1/q'$ we also need
\begin{proposition} \label{Tkmn-seqmod}
Let $1<q\leq p <\infty$ and $Z=Z_N(E)$ with $E$ as in \eqref{sizeassumption}.

(i) For $m\ge 0$,
\begin{multline*}\Big\|\Big(\sum_{k}\sum_{0\le n\le N} |2^{(m+n)/q'}T^k_{m,n} U_{k+m+n}g_{k+m+n}|^q\Big)^{1/q}\Big\|_p \\ \lc Z^{1/q}
2^{-m(\frac 1p-\eps)} \|\vec g\|_{L_p(\ell_q)}.
\end{multline*}

(ii) For $m\le 0$,
$$\Big\|\Big(\sum_{k}\sum_{0\le n\le N} |2^{n/q'}T^k_{m,n} U_{k+m+n}g_{k+m+n}|^q\Big)^{1/q}\Big\|_p \lc Z^{1/q} \|\vec g\|_{L_p(\ell_q)}.$$
\end{proposition}
\begin{proof}
For the proof of (i) we interpolate between
\Be\label{q=pest}
\Big\|\Big(\sum_{k}\sum_{0\le n\le N} |2^{(m+n)/q'}T^k_{m,n} U_{k+m+n}g_{k+m+n}|^q\Big)^{1/q}\Big\|_q \lc Z^{1/q}2^{-m/q}
\|\vec g\|_{L_q(\ell_q)}
\Ee
and
\Be\label{largepest}
\Big\|\Big(\sum_{k}\sum_{0\le n\le N} |2^{(m+n)/q'}T^k_{m,n} U_{k+m+n}g_{k+m+n}|^q\Big)^{1/q}\Big\|_p \lc Z^{1/q}
\|\vec g\|_{L_p(\ell_q)}
\Ee
which we use for large $p$.

Recall $T^k_{m,n}=0$ if $2^{k+m}\notin\HF(E)$. To see \eqref{q=pest} we interchange summation and integration and use Lemma
\ref{Tmnq} for $\rho=q$ to estimate the left hand side by
\begin{align*}
&\Big(\sum_{k:2^{k+m}\in \HF(E)}
\sum_{0\le n\le N} [2^{-m/q}\|g_{k+m+n}\|_q]^q\Big)^{1/q}
\\&\lc 2^{-m/q}\Big(\sum_l\|g_l\|_q^q \#\{n: 0\le n\le N,\,\,2^{l-n}\in \HF(E)\}\Big)^{1/q}
\\
&\lc 2^{-m/q}Z^{1/q}\|\vec g\|_{L_q(\ell_q)}\,.
\end{align*}

To see \eqref{largepest} we use Proposition \ref{vectpeetremodprop}.
By Lemma \ref{Tkmnptwise}, (i), we have that the left hand side of \eqref{largepest} is dominated by
\begin{align*}
&\Big\|\Big(\sum_{k:2^{k+m}\in \HF(E)}\sum_{0\le n\le N} |2^{-(m+n)/q} \cM^{k+n+m}_{n+m} U_{k+n+m}g_{k+m+n}|^q\Big)^{1/q}\Big\|_p
\\
&\lc\Big\|\Big(\sum_{k:2^{k+m}\in \HF(E)}\sum_{0\le n\le N} |U_{k+n+m}g_{k+m+n}|^q\Big)^{1/q}\Big\|_p
\\
&\lc\Big\|\Big(\sum_{k:2^{k+m}\in \HF(E)}\sum_{0\le n\le N} |g_{k+m+n}|^q\Big)^{1/q}\Big\|_p
\\
&\lc\Big\| \Big(\sum_l |g_l|^q
\#\{n: 0\le n\le N,\,2^{l-n}\in \HF(E)\}\Big)^{1/q}
 \Big \|_p
\end{align*}
and  the last expression is $\lc Z^{1/q} \|\vec g\|_{L_p(\ell_q)}$.
This concludes the proof of  (i).

For the proof of (ii) we use Lemma \ref{Tkmnptwise}, (iii) and
again Proposition \ref{vectpeetremodprop} to
see that
\begin{align*}
&\Big\|\Big(\sum_{k}\sum_{0\le n\le N} |2^{n/q'}T^k_{m,n} U_{k+m+n}g_{k+m+n}|^q\Big)^{1/q}\Big\|_p
\\
&\lc \Big\|\Big(\sum_{k: 2^{k+m}\in \HF(E)}\sum_{0\le n\le N} |2^{-n/q}
\cM^{k+m+n}_nU_{k+m+n}g_{k+m+n}|^q\Big)^{1/q}\Big\|_p
\\
&\lc \Big\|\Big(\sum_{k: 2^{k+m}\in \HF(E)}\sum_{0\le n\le N} |g_{k+m+n}|^q\Big)^{1/q}\Big\|_p
\\&\lc Z^{1/q} \|\vec g\|_{L_p(\ell_q)} \,.\qedhere
\end{align*}
\end{proof}

\begin{proof}[Proof of Theorem \ref{upperlabounds}]


By the triangle inequality in $L_p(\ell_q)$ we have
\begin{multline}\label{Mink-mn}
\begin{aligned}
&\Big\|\Big(\sum_{k\in \bbN} 2^{ksq}\big|\psi_k *P_E\big [{\textstyle \sum_{l=0}^\infty 2^{-ls} \psi_l*f_l} ]\big|^q\Big)^{1/q}\Big\|_p\\
&\le
\sum_{m\ge 0} 2^{-sm} \big[ I_m+II_m+III_m+IV_m]
+\sum_{m<0}2^{-sm}\big[V_m+VI_m] \\ &\qquad + \sum_{m\ge 0} \sum_{n<
-m}2^{-s(n+m)}VII_{m,n}+\sum_{m<0}\sum_{n<0}2^{-s(n+m)}VII_{m,n}\,,
\end{aligned}
\end{multline}
where for $m\ge 0$,
\begin{subequations}\label{roman}
\begin{align}
I_m&= \Big\|
\Big(\sum_{\substack {k\ge0\\ k+m\in\HF(E)} }\big|\sum_{n>\max\{N-m,0\}}
2^{-sn}T^k_{m,n} f_{k+m+n}\big|^q\Big)^{1/q}\Big\|_p,
\\
II_m&= \Big\|
\Big(\sum_{\substack {k\ge0\\ k+m\in\HF(E)} }\big|\sum_{0\le n\le \max\{N-m,0\}}2^{-sn}T^k_{m,n} f_{k+m+n}\big|^q\Big)^{1/q}\Big\|_p,
\\
III_m&= \Big\|
\Big(\sum_{\substack {k\ge0\\ k+m\in\HF(E)}} \big|\sum_{\substack{n\le 0\\m+n\ge N}}2^{-sn}T^k_{m,n} f_{k+m+n}\big|^q\Big)^{1/q}\Big\|_p,
\\
IV_m&= \Big\|
\Big(\sum_{\substack {k\ge0\\ k+m\in\HF(E)} }\big|\sum_{\substack{n \le 0\\0\le
m+n\le N}}2^{-sn}T^k_{m,n} f_{k+m+n}\big|^q\Big)^{1/q}\Big\|_p,
\end{align}
and, for $m\le 0$,
\begin{align}
V_m&= \Big\|
\Big(\sum_{\substack {k\ge0\\ k+m\in\HF(E)} }\big|\sum_{n> N} 2^{-sn}T^k_{m,n} f_{k+m+n}\big|^q\Big)^{1/q}\Big\|_p,
\\
VI_m&= \Big\|
\Big(\sum_{\substack {k\ge0\\ k+m\in\HF(E)} }\big|\sum_{0\le n\le N} 2^{-sn}T^k_{m,n} f_{k+m+n}\big|^q\Big)^{1/q}\Big\|_p.
\end{align}
Moreover
\Be
VII_{m,n}=
 \Big\|\Big(\sum_{\substack {k\ge0\\ k+m\in\HF(E)} }\big|T^k_{m,n} f_{k+m+n}\big|^q\Big)^{1/q}\Big\|_p,  \quad n<\min\{-m,0\}.
\Ee
\end{subequations}
When $-1/p'<s<-1/q'$ we estimate the terms $I_m,\dots,VI_m$ by another use of
the triangle  inequality in $L_p(\ell_q)$, with respect to the  $n$ summation.
When
 $s=-1/q'$ we still do this for the terms $I_m$, $V_m$ but argue differently
for the terms involving the restriction $0\le n\le N$.
In what follows we shall need to distinguish the cases
$s<-1/q'$ and $s=-1/q'$ in various estimates and therefore write
$I(s)$, $II(s)$,...  for  the expressions $I$, $II$, ..., resp.

By Proposition \ref{Tkmn-seq}, (i)  we have
\begin{align*}
&\sum_{m\ge 0} 2^{-sm}I_m (s)\lc \big \|\vec f\big\|_{L_p(\ell_q)} \times
\\& 2^{N(\frac 1q-\frac 1p)}
\Big[
\sum_{m\ge N}2^{-(1+s)m}\sum_{n\ge 0}2^{-n(s+\frac 1{p'})}
+
\sum_{0\le m< N}2^{-(1+s)m}\sum_{n\ge N-m}2^{-n(s+\frac 1{p'})}\Big]
\end{align*} and the constant for $q\le p<\infty$  is easily seen to be
$O(2^{N(-s-1/q')})$.

Next we have $II_m(s)=0$ for $m>N$. For the terms with $0\le m\le N$, we get  again by
Proposition \ref{Tkmn-seq}, (i),
\begin{align*}
&\sum_{0\le m\le N} 2^{-sm} II_m (s)\\&\lc
\sum_{0\le m\le N} 2^{-m(s+1)} 2^{m(\frac 1q-\frac 1p-\eps)}\sum_{0\le n\le N-m}
2^{-n (s+\frac 1{q'})}
\big \|\vec f\big\|_{L_p(\ell_q)}\\&
\lc \max\{|s+1/q'| 2^{-N (s+\frac 1{q'})}, N \} \big \|\vec f\big\|_{L_p(\ell_q)}
\end{align*}
which contributes the desired bound for $-1/p'<s<-1/q'$.
For $s=-1/q'$ we use H\"older's inequality in the  $n$-sum followed by Proposition
\ref{Tkmn-seqmod} to get
\begin{align*}&\sum_{0\le m\le N} 2^{m/q'}II_m(-1/q')
\\&\lc \sum_{0\le m\le N} 2^{m/q'}N^{1/q'}  \Big\|
\Big(\sum_{\substack {k\ge0\\ k+m\in\HF(E)} }\sum_{\substack{0\le n \le\\ \max\{N-m,0\}}}\big|2^{n/q'}T^k_{m,n} f_{k+m+n}\big|^q\Big)^{1/q}\Big\|_p
\\
&\lc N^{1/q'}\sum_{m\ge 0} 2^{-m(\frac 1p-\eps)}  Z^{1/q} \|\vec f\|_{L_p(\ell_q)}\lc
N^{1/q'}Z^{1/q} \|\vec f\|_{L_p(\ell_q)}\,.
\end{align*}

For $-1/p'<s\le -1/q'$ we have by Proposition \ref{Tkmn-seq}, (ii),
\begin{align*}
\sum_{m\ge 0} 2^{-ms}III_m(s)&\le\sum_{m\ge N} \sum_{N-m\le n\le 0}2^{-(m+n)s}
2^{n-m} 2^{N(\frac 1q-\frac 1p)}\|\vec f\|_{L_p(\ell_q)}
\\&\lc 2^{-N/p} 2^{-N(s+1-1/q)} \|\vec f\|_{L_p(\ell_q)}\,.
\end{align*}

Similarly,
\begin{align*}
&\sum_{m\ge 0} 2^{-ms}IV_m(s)\\
&\lc\sum_{m\ge 0} \sum_{\substack{-m \le n\le \\ \min\{N-m,0\}}}
2^{-(m+n)s} 2^{n-m} 2^{(m+n)(\frac 1q-\frac 1p+\eps)} \|\vec f\|_{L_p(\ell_q)}
\\\,&
+\sum_{m\ge N}\sum_{N-m\le n\le 0} 2^{-(m+n)s} 2^{n-m} 2^{N(\frac 1q-\frac 1p)}
\|\vec f\|_{L_p(\ell_q)}
\end{align*}
with the implicit constant depending on $\eps>0$. Since $p<\infty$ we may choose
$0<\eps<1/p$.
We evaluate various geometric series and obtain for $s\le -1/q'$
\[\sum_{m\ge 0} 2^{-ms}IV_m(s)\lc2^{-N(\frac 1p-\eps)}  2^{-N(s+1-\frac 1q)} \|\vec f\|_{L_p(\ell_q)}.\]

Next we consider the terms with $m\le 0$. We use the second estimate
in Proposition \ref{Tkmn-seq}, (iv),
and $s>-1/p'$ to obtain
\begin{align*}
\sum_{m\le 0}2^{-sm}V_m(s)
&\lc2^{N(\frac 1q-\frac 1p)}\sum_{m\le 0} 2^{-sm}\sum_{n\ge N} 2^{-n(s+1-\frac 1p)}
\|\vec f\|_{L_p(\ell_q)}
\\
&\lc 2^{N(-s-1+\frac 1q)} \|\vec f\|_{L_p(\ell_q)}\,.
\end{align*}
For the terms $VI_m(s)$  we
need to separately treat the case $s<-1/q'$ and $s=-1/q'$. For $s<-1/q'$ we
use the first  estimate in Proposition \ref{Tkmn-seq}, (iv) and
get
\begin{align*}
\sum_{m\le 0}2^{-sm}VI_m(s)
&\lc\sum_{m\le 0}2^{-sm} \sum_{0\le n\le N} 2^{n(-s-1+\frac 1q)}
\|\vec f\|_{L_p(\ell_q)}
\\
&\lc 2^{N(-s-1+\frac 1q)} \|\vec f\|_{L_p(\ell_q)}\,.
\end{align*}
For $s=-1/q'$ we argue as for the $II_m(-1/q')$ terms above and use H\"older's inequality in the  $n$ sum followed by Proposition
\ref{Tkmn-seqmod} (ii)to get
\begin{align*}&\sum_{m\le 0} 2^{m/q'}VI_m(-1/q')
\\&\lc \sum_{m\le 0} 2^{m/q'}N^{1/q'}  \Big\|
\Big(\sum_{\substack {k\ge0\\ k+m\in\HF(E)} }\sum_{0\le n\le N}\big|2^{n/q'}T^k_{m,n} f_{k+m+n}\big|^q\Big)^{1/q}\Big\|_p
\\
&\lc N^{1/q'}\sum_{m\le 0} 2^{m/q'} Z^{1/q} \|\vec f\|_{L_p(\ell_q)}\lc
N^{1/q'}Z^{1/q} \|\vec f\|_{L_p(\ell_q)}\,.
\end{align*}

Finally, the inequalities
\[\sum_{m\ge 0} \sum_{n<-m}
2^{-(m+n)s}VII_m(s)\lc\|\vec f\|_{L_p(\ell_q)}\] and
\[\sum_{m\le 0}\sum_{n\le 0}
2^{-(m+n)s}VII_m(s)\lc\|\vec f\|_{L_p(\ell_q)}\]
follow immediately from Proposition \ref{Tkmn-seq}, (iii), (v), resp.
\end{proof}

\section{Bounds for families of test functions}
\label{testfctsect}

It will be convenient to use characterizations of function spaces by compactly
supported localizations (i.e. the  local means in \cite{triebel2}), see \S
\ref{LPsect}. In what
follows let $M_0$, $M_1$ be positive integers, and we shall always assume that
$-M_0<s<M_1$.

Let $\psi_0$, $\psi$ be $C^\infty$ functions supported in $(-1/2,1/2)$ so that
$\widehat \psi_0(\xi) \neq 0$ for $|\xi|\le 1$ and so that $\psi(\xi)\neq 0$
for $1/4\le |\xi|\le 1$, moreover $\widehat \psi$ vanishes of order $M_1$ at
$0$. Thus the cancellation condition \eqref{psicanccond} holds.
Let $\psi_k=2^k\psi(2^k\cdot)$ for $k=1,2,\dots$.
We shall use the characterization of $F^s_{p,q}$ using the $\psi_k$, see
\eqref{localmeans}.
Now we will define some test functions which will be used to establish the
lower bounds in Theorem \ref{sharpGabd}. In what follows we fix an integer $m\ge
0$; all implicit constants will be independent of $m$.

Let $\eta$ be a $C^\infty$ function supported in $(-1/2,1/2)$ such that
$\int\eta(x) x^n dx=0$ for $n=1,\dots, M_0$. Let, for $l\ge m$,  $\cP_l^m$ be a
set of $2^{m-l}$-separated points in $[0,1]$. That is
$\cP_l^m=\{x_{l,1},\dots, x_{l,N(l)}\}$ with $N(l)\le  2^{l-m}$ and $x_{l,\nu}<
x_{l,\nu+1}$ with $x_{l,\nu+1}- x_{l,\nu}\ge 2^{m-l}$.
Define
\Be\label{etalnu} \eta_{l,\nu}=\eta(2^l(x-x_{l,\nu})).\Ee

Let $\fL^m$ be a finite set of nonnegative integers $\ge m$ and assume $\#
\fL^m\ge 2^m$.
Let
\begin{subequations} \label{fS}
\Be\fS^m=\{(l,\nu): l\in \fL^m, \,x_{l,\nu}\in \cP^m_l\};\Ee
and \Be\fS_{l}^m=\{\nu: (l,\nu)\in \fS^m\}\Ee
\end{subequations}
For  any  indexed sequence
$\{a_{l,\nu}\}$ satisfying $\sup_{l,\nu}|a_{l,\nu}|\le 1,$
we define for $l\in \fL^m$,
\begin{equation}
g_m(x)= \sum_{l\in \fL^m} 2^{-ls}
 \sum_{\nu\in \fS^m_l}
a_{l,\nu}\eta_{l,\nu}(x)\,.
\end{equation}
If the families $\fL^m$, $m=1,2,\dots, $ are disjoint, we set
$$
    g:=\sum\limits_{m\geq 1} \beta_m g_m\,.
$$
The proof of the following proposition is a modification of the proof of  a
corresponding result by Christ and one of the authors (\cite{cs-lms}).
\begin{proposition}\label{testfctprop} Let $s>-M_0$.\\
{\em (i)} If $1\leq p,q<\infty$ then
$$
    \|g_m\|_{F^s_{p,q}}\lesssim_{p,q,s}\Big\|\Big(\sum\limits_{l \in
\fL^m}\Big|\sum\limits_{\nu\in \fS^m_l}a_{l,\nu}\bbone_{l,\nu}
\Big|^q\Big)^{1/q}\Big\|_p
$$
and
$$
\|g\|_{F^s_{p,q}}\lesssim_{p,q,s}\Big\|\Big(\sum\limits_{m}
|\beta_m|^q\sum\limits_ { l \in
\fL^m}\Big|\sum\limits_{\nu\in \fS^m_l}a_{l,\nu}\bbone_{l,\nu}
\Big|^q\Big)^{1/q}\Big\|_p\,.
$$
{\em (ii)} If $1\le q\le p<\infty$ then there exists $C_{p,q,s}$ such that
\Be\label{f000}
\|g_m\|\ci{F_{p,q}^s} \le
C_{p,q,s} \big(2^{-m} \#(\fL^m)\big)^{1/q}
\Ee
and
\Be\label{betamineq}
\|g\|\ci{F^s_{p,q}} \le
C_{p,q,s} \Big(\sum_{m\ge 1} |\beta_m|^q 2^{-m} \#(\fL^m)\Big)^{1/q}.
\Ee
\end{proposition}

\begin{proof} The functions $\{\eta_{l,\nu}\}_{l,\nu}$ represent a family
of ``smooth atoms'' in the sense of Frazier/Jawerth \cite[Thm.\ 4.1
and \S 12]{FrJa90} which immediately implies the relations in (i). Here we
need the pairwise disjointness of the sets $\fL^m$.

We continue with proving
\Be\label{intervalbd}
 \Big\|\Big(\sum_{l\in \fL^m}
\Big|\sum_{\nu\in \fS_{l}^m} \bbone\ci{I_{l,\nu}}
\Big|^q\Big)^{1/q}\Big\|_p
\lc\ci{p,q} (2^{-m}\#(\fL^m))^{1/q}\,.
\Ee
Then (i) together with \eqref{intervalbd} and $\sup_{l,\nu}|a_{l,\nu}|\le 1$
gives \eqref{f000}\,.

Indeed, let
$G_l(x) =\sum_{\nu\in \fS_{l}^m} \bbone\ci{I_{l,\nu}}(x)$
and $G(x)=\Big(\sum_{l\in \fL^m} G_l(x)^q \Big)^{1/q}.$ In order to control
the $L_p$-norm of $G$ we use the dyadic version of the Fefferman-Stein
interpolation theorem for $L_q$ and $BMO$. Note that  the proof of  \cite[Thm.\
5]{feffstein-Hp} gives the dyadic version of $\#$-function estimate and thus one
can work with the $BMO_{\rm dyad}$ norm in \cite[Cor.\ 2]{feffstein-Hp}.
Consequently it suffices to show that the norms of $G$ in $L_q$ and $BMO_\dyad$
are bounded by $C(2^{-m} \#(\fL^m))^{1/q}$.
This is immediate for the $L_q$ norm. For the $BMO_{\dyad}$ norm we have to show
that
\Be\label{dyadbmo}
\sup_J\inf_c \frac{1}{|J|}\int_J |G(y) -c| dy \lc (2^{-m} \#(\fL^m))^{1/q}\,,
\Ee
where the sup is taken over all dyadic intervals and the inf is taken over all complex numbers.

For a fixed dyadic interval $J$ with midpoint $x_J$  we define
$$c_{J,l}= \begin{cases}
\sum_{\nu\in \fS_{l}^m} \bbone\ci{I_{l,\nu}}(x\ci J) &\text{ if } 2^{-l}\ge |J|
\\
0 &\text{ if } 2^{-l}< |J|\end{cases}
$$
and
$$c_J = \Big(\sum_{l\in \fL^m} c_{J,l}^q \Big)^{1/q}
.$$

Fix $J$. Then
\begin{align*}
\frac{1}{|J|}\int_J |G(y) -c_J| dy
&\le
\frac{1}{|J|}\int_J \Big|
\Big(\sum_{l\in \fL^m} G_l(y)^q \Big)^{1/q} -
\Big(\sum_{l\in \fL^m} c_{J,l}^q \Big)^{1/q}\Big| dy
\\
&\le\frac{1}{|J|}\int_J
\Big(\sum_{l\in \fL^m} |G_l(y)-c_{J,l}|^q \Big)^{1/q}  dy
\\
&\le\Big(\sum_{l\in \fL^m}
  \frac{1}{|J|}\int_J |G_l(y)-c_{J,l}
|^q dy\Big)^{1/q}\,.
\end{align*}

Here we have used the triangle inequality in $\ell_q$ and
H\"older's inequality on the interval $J$.
Note that
$$
G_l(y)=c\ci{J,l} \text{ if $y\in J$ and $2^{-l}\ge |J|$.}
$$
Since $c_{J,l}=0$ if
 $2^{-l}< |J|$ we get from the previous estimate
\Be
\frac{1}{|J|}\int_J |G(y) -c\ci J| dy
\le
\Big(\sum_{\substack {l\in \fL^m\\ 2^{-l}<|J|}}
  \frac{1}{|J|}\int_J |G_l(y)|^q dy
\Big)^{1/q}  .
\Ee
Now  by the definition of $G_l$ we have
$$
\int_J |G_l(y)|^q dy \le
\begin{cases}
2^{-l} &\text{ if } 2^{-m}|J|\le 2^{-l}< |J|
\\
2^{-m} |J|
 &\text{ if } 2^{-l} <2^{-m}|J|
\end{cases}
$$
and thus
\begin{align*}
\sum_{\substack {l\in \fL^m\\ 2^{-l}<|J|}}
  \frac{1}{|J|}\int_J |G_l(y)|^q dy\,&\le\,
\sum_{l: 2^{-m}|J|\le 2^{-l}< |J|} (2^l|J|)^{-1}
+\sum_{l\in \fL^m}
2^{-m}
\\
&\lc (1+ 2^{-m} \#(\fL^m))\,.
\end{align*}
Since we assume that
$ \#(\fL^m)\ge 2^m$ this finishes the proof of \eqref{dyadbmo}.

Finally, \eqref{betamineq} is a consequence of the second relation in (i), the
triangle inequality in $L_{p/q}$ and \eqref{intervalbd}. In fact, the second
relation in (i) can be rewritten to
$$
    \|g\|_{F^s_{p,q}} \lesssim \Big\|\sum\limits_{m\geq 1}
|\beta_m|^q\sum\limits_{l\in
\fL^m}\Big|\sum\limits_{\nu \in \fS^m_l}a_{\ell,\nu}\bbone_{l,\nu}
\Big|^q\Big\|_ {p/q}^{1/q}
$$
Since $p/q \geq 1$ we obtain
$$
\|g\|_{F^s_{p,q}} \lesssim
\Big(\sum\limits_{m}
|\beta_m|^q\Big\|\sum\limits_{l \in
\fL^m}\Big|\sum\limits_{\nu \in \fS^m_l}\bbone_{l,\nu}\Big|^q
\Big\|_{p/q}\Big)^{1/q}
$$
and \eqref{intervalbd} finishes the proof.
\end{proof}

\section{Lower bounds for Haar projection numbers}
\label{lowerbdsect-mosts}

In this section we require that  $\psi$ is supported on
$(-2^{-4}, 2^{-4})$ and that  $\int \psi(x) x^Mdx=0$ for $M=0,1,\dots, M_0$ for some large integer $M_0$, and let $\psi_k=2^k\psi(2^k\cdot)$.
Let $\Psi(x)= \int_{-\infty}^x \psi(t) dt$, the primitive which is also supported in   $(-2^{-4}, 2^{-4})$.
For $h_{0,0}=\chi_{[0,1/2)}-\chi_{[1/2,1)}$ we have
$$\psi*h_{0,0} (x) = \Psi(x)+\Psi(x-1)-2\Psi(x-\tfrac 12)$$
and therefore $\psi*h_{0,0} (x) = -2\Psi(x-\tfrac 12) $ for $ x\in [1/4,3/4].$
Thus there is $c_0>0$ and  a subinterval $J\subset [1/4,3/4]$ so that
$$|\psi* h_{0,0}(x)|\ge  c_0, \text{ for } x\in J\,.$$
For $k=0,1,2,\dots$ and $\mu\in \bbZ$ let $J_{k,\mu}=2^{-k}\mu +2^{-k}J$ which is a subinterval of the middle half  of $I_{k,\mu}$, of length $\gc 2^{-k}$,  and we have
\Be \label{lowerboundforconv}
|\psi_k*h_{k,\mu}(x)|\ge c_0 \text{ for } x\in J_{k,\mu}.
\Ee

We now prepare for the definition of a  suitable family of test functions.
Let $\eta$ be an odd $C^\infty$ function supported in $(-2^{-5},2^{-5})$
so that $\int \eta(x) x^Mdx=0$ for $M=0,1,\dots, M_0$ and so that
\Be \label{etanondeg}
2\int_{0}^{1/2} \eta(x) dx =
\int_{0}^{1/2} \eta(x) dx -  \int_{-1/2}^0 \eta(x) dx
\ge 1\,.
\Ee

We now pick an arbitrary set $A$ of Haar frequencies and $N$ so that
\Be\label{cardA}\Lambda<\#A+1, \text{ and } 2^N\le \# A<2^{N+1},
\Ee
and fix $N$ for the remainder of the section.
Define, for $n=N, N-1,\dots, 1$
\Be
\eta_{k,n,\mu}(y) =\eta(2^{k+n}(x-2^{-k}\mu-2^{-k-1}))
\label{etakm}
\Ee
Let  $r_k$ denote the Rademacher function on $[0,1]$.
For $t\in [0,1]$ and $2^k\in A$ let
\begin{subequations}\label{Ups}
\Be\label{Upsk}
\Ups_{k}(y)=
\sum_{n=0}^N \alpha_n\Ups_{k,n}(y) \Ee  with \Be\label{Upskm} \Ups_{k,n}(y)= 2^{n(-s+1/q)}\sum_{\mu=0}^{2^k-1}
\eta_{k,n,\mu}(y).
\Ee
\end{subequations}
\begin{subequations}\label{fteq}
Let
\Be\label{ftn}f_{n,t}(y)=2^{-N/q} \sum_{2^k\in A} r_k(t) 2^{-ks}\Ups_{k,n}(y)\Ee and
\Be\label{ft}f_t(y)=\sum_{n=1}^N\alpha_n f_{n,t}.\Ee
\end{subequations}
\begin{lemma} The following estimates hold uniformly in $t\in [0,1]$.

(i) For $n=1,\dots, N$,
$$\|f_{n,t} \|_{F^s_{p,q}}\le C_{p,q,s}.$$
(ii) Suppose that $\log_2 A$ is $N$-separated (i.e. $2^j\in A$, $2^{j'}\in A$, $j\neq j'$ implies $|j-j'|\ge N$). Then
$$\|f_{t} \|_{F^s_{p,q}}\le C_{p,q,s}
\Big(\sum_{n=1}^N|\alpha_n|^q\Big)^{1/q}.
$$
\end{lemma}
\begin{proof} Let $\fL^n=\{l:2^{l-n}\in A\}$. Then
$$f_{n,t}(y)= 2^{(n-N)/q} \sum_{l\in \fL^n} r_{l-n}(t) 2^{-ls}
\sum_\mu \eta(2^l(x-2^{n-l}\mu- 2^{n-l-1})$$
and (i) follows from Proposition \ref{testfctprop} since $2^{-n}\#(\fL^n) \lc
2^{N-n}$. Since the sets $\fL^n$, $n=1,\dots, N$, are essentially disjoint (ii)
follows as well.
\end{proof}
For $t\in [0,1]$ let
\Be T_t g(x)=\sum_{2^j\in A} r_j(t)   \sum_{\mu=0}^{2^j-1}2^j
\inn{h_{j,\mu} }{g} h_{j,\mu}(x).
\Ee
We seek to derive a lower bound for $\|T_{t_1}f_{t_2}\|_{F^s_{p,q}}$ for most
$t_1, t_2$. This is accomplished by
\begin{proposition} \label{qlowerbd}
Let $-1<s\le -1/q'$. Let $f_{N,t}$ as in \eqref{ftn} (with $n=N$) and $f_t$  as in \eqref{ft}.
 Then  there is $c>0$ such that the following relations\\
\noindent (i) $$
\Big(\int_{0}^1\int_0^1 \big\| T_{t_1} f_{N,t_2} \big\|\ci{F^s_{p,q}}^q dt_1 dt_2\Big)^{1/q} \ge c  2^{N(-s-1/q')}
$$ and (ii) $$
\Big(\int_{0}^1\int_0^1 \big\| T_{t_1} f_{t_2} \big\|\ci{F^s_{p,q}}^q dt_1 dt_2\Big)^{1/q} \ge
c \Big| \sum_{n=0}^N\alpha_n 2^{n(-s-\frac{1}{q'})}\Big|
$$
hold true. 
\end{proposition}
\begin{proof}
Note that (i) is a special case of (ii) (with the choice $\alpha_N=1$, $\alpha_n=0$ for $n\le N$).
The left hand side in (ii)  is equivalent with
$$\Big(\int_{0}^1\int_0^1 \Big\| \Big(\sum_{k=0}^\infty 2^{ksq} |\psi_k* T_{t_1} f_{t_2}|^q \Big)^{1/q} \Big\|_p^q dt_1 dt_2\Big)^{1/q}. $$
Since  $\psi_k* T_{t_1}f_{t_2}$ is supported in $[-1,2]$, we can use
 H\"older's inequality (with $p\ge q$) to see that this expression is bounded below by a positive constant times
\begin{align}&\Big(\int_{0}^1\int_0^1 \Big\| \Big(\sum_{2^k\in A} 2^{ksq} |\psi_k*( T_{t_1} f_{t_2})|^q \Big)^{1/q} \Big\|_q^q dt_1 dt_2\Big)^{1/q}
\notag
\\
\label{fubini}&=
\Big(\sum_{2^k\in A}2^{ksq}
\Big\|\Big(\int_{0}^1\int_0^1 \big|
\psi_k* T_{t_1} f_{t_2}(x)\big|^q dt_1 dt_2\Big)^{1/q}
\Big\|_q^q \Big)^{1/q}.
\end{align}
For fixed $x$ we have
$$\psi_k*T_{t_1} f_{t_2}(x)= 2^{-N/q}\sum_{2^j\in A}\sum_{2^l\in A} r_j(t_1)r_l(t_2)
2^{-ls}\sum_{\mu=0}^{2^j-1}2^j\inn{\Ups_{l}}{h_{j,\mu} }\psi_k* h_{j,\mu}(x)
$$
and, by Khinchine's inequality,
\begin{align*}
&\Big(\int_{0}^1\int_0^1 \big|
\psi_k*( T_{t_1} f_{t_2})(x)\big|^q dt_1 dt_2\Big)^{1/q}
\\&\ge c(q) 2^{-N/q} \Big(\sum_{2^j\in A}\sum_{2^l\in A}\Big|
2^{-ls}
\sum_{\mu=0}^{2^j-1}2^j\inn{\Ups_{l}}{h_{j,\mu} }\psi_k* h_{j,\mu}(x)\Big|^2\Big)^{1/2}
\end{align*}
Hence, for $2^k\in A$, picking up only the terms with $j=k$ and $l=k$,
\begin{align*}
&\Big(\int_{0}^1\int_0^1 \Big|
\psi_k*( T_{t_1} f_{t_2})(x)\Big|^q dt_1 dt_2\Big)^{1/q}
\\&\gc 2^{-N/q} \Big|
2^{-ks}
\sum_{\mu=0}^{2^k-1}2^k\inn{\Ups_{k}}{h_{k,\mu} }
\psi_k* h_{k,\mu}(x)\Big|\,.
\end{align*}
Observe that the supports of $h_{k,\mu}$ and $\eta_{k,m,\mu'}$ are disjoint when $\mu\neq\mu'$. Thus
\begin{align*}
2^k\inn{\Ups_{k}}{h_{k,\mu} }&=
\sum_{n=0}^N \alpha_n2^{n(-s+1/q)}\sum_{\mu'=0}^{2^k-1}
2^k\inn{\eta_{k,n,\mu'}} {h_{k,\mu}}
\\
&=
\sum_{n=0}^N \alpha_n 2^{n(-s+1/q)}
2^k\inn{\eta_{k,n,\mu} }{h_{k,\mu}}\,.
\end{align*}
Furthermore
\begin{align*}
&2^k\inn{\eta_{k,n,\mu}} {h_{k,\mu}}
\\
&= 2^k \int \eta(2^{k+n}(x-2^{-k}\mu-2^{-k-1})) h_{0,0}(2^k x-\mu) dx
\\
&=  \int \eta (2^n(y-\tfrac 12)) h_{0,0} (y) dy
\\
&=  \int_{-1/2}^{0} \eta (2^n y) dy
-\int_0^{1/2} \eta (2^n y) dy \\
&= -
2^{-n+1} \int_0^{1/2} \eta(u) du\,
\end{align*}
where in the last line we have used that  $\eta$ is odd and supported in $(-2^{-4}, 2^{-4})$.

Next we observe that
$$\psi_k*h_{k,\tilde\mu}(x)=0, \text{ for  $x\in J_{k,\mu}$, $\mu\neq \tilde\mu$. }$$
So from the above we get, for $x\in J_{k,\mu}$
\begin{align*}
&2^k \sum_{\tilde\mu=0}^{2^k-1}\inn {\Ups_k}{h_{k,\tilde\mu} }\psi_k*h_{k,\tilde\mu}(x)
\\&=2^k \inn {\Ups_k}{h_{k,\mu} }\psi_k*h_{k,\mu}(x)
\\
&=-2\psi_k*h_{k,\mu}(x) \int_0^1\eta(u)du\, \sum_{n=0}^N \alpha_n 2^{n(\frac
1q-1-s)} .
\end{align*}
Finally  we can prove the lower bound for the expression
\eqref{fubini} and obtain
\begin{align*}
&\Big(\sum_{2^k\in A}2^{ksq}
\Big\|\Big(\int_{0}^1\int_0^1 \big|
\psi_k*( T_{t_1} f_{t_2})(x)\big|^q dt_1 dt_2\Big)^{1/q}
\Big\|_q^q \Big)^{1/q}
\\
&\gc \Big(\sum_{2^k\in A} 2^{ksq} \sum_{\mu=0}^{2^k-1}\int_{J_{k,\mu}}
\Big[2^{-N/q} 2^{-ks} 2^k|\inn{\Ups_k}{h_{k,\mu}}| |\psi_k*h_{k,\mu}(x)| \Big]^q dx \Big)^{1/q}
\\
&\gc
\Big| \int_0^{1/2} \eta(x) dx   \sum_{n=0}^N \alpha_n2^{n(-s-\frac 1{q'})}\Big|
\Big(\sum_{2^k\in A}2^{-N} \sum_{\mu=0}^{2^k-1}\int_{J_{k,\mu}} |\psi_k*h_{k,\mu}(x)|^q dx \Big)^{1/q}
\\
&\gc
 \Big| \sum_{n=0}^N \alpha_n 2^{n(-s-\frac 1{q'})} \Big|\,.
  \end{align*}
  Here we have used
  \eqref{lowerboundforconv}, \eqref{etanondeg}, and \eqref{cardA}, and the condition  $s\le -1/q'$.
\end{proof}

\subsection*{\it Growth of  $\gamma_*(F^s_{p,q}, \Lambda)$, $s<-1/q'$}
Take $A$ as in \eqref{cardA}.
Let  $f_{t,N}$ be as in \eqref{ftn} with $n=N$, so that $\|f_{t,N}\|_{F^s_{p,q}}\lc 1$.
By Proposition \ref{qlowerbd} there exist $t_1$, $t_2$ in $[0,1]$ so that
$$\|T_{t_1} f_{N,t_2} \|_{F^s_{p,q}}\gc 2^{N(-s-\frac 1{q'})}$$
Hence
$$\|T_{t_1}  \|_{F^s_{p,q} \to F^s_{p,q}}\ge c_{p,q,s}
2^{N(-s-\frac 1{q'})}$$
Now let
\Be \label{Epm}E^{\pm}:=\{h_{j,\mu}: 2^j\in A,\, r_j(t_1)=\pm 1, \,\mu=0, \dots, 2^j-1\}.\Ee Then
$$T_{t_1} = P_{E^+}- P_{E^-}$$ and thus at least one of
$ P_{E^+}$ or $P_{E^-}$ has operator norm bounded below by
$c_{p,q,s}2^{N(-s-\frac 1{q'})}$.
Since $\HF(E^\pm)\subset A$ we get
$$\cG (F^s_{p,q} , A)
\gc 2^{N(-s-\frac 1{q'})}, \qquad s\le -1/q'$$
and the asserted lower bound for
$\cG (F^s_{p,q} ;\Lambda)$ follows in the range $s<-1/q'$.
By  Theorem \ref{upperlabounds} we also have
$$\cG(F^{-1/q'}_{p,q},A)\le c(p,q,s) \Lambda^{-s-1/q'}.$$ Thus
$\gamma^*(F^{-s}_{p,q};\La)\approx
\Lambda^{-s-1/q'}$
for large $\La$.

\subsection*{\it Remark}
The above arguments already  give a lower bound $c(\log \Lambda)^{1/q'}$  in the endpoint case,
for the lower Haar projection numbers
 $\gamma_*(F^{-1/q'}_{p,q}, \Lambda)$. Let $A'$ be an $2N$ separated subset
 with $\#(A')\ge (2N)^{-1}\#A$. Let $f_t$ as in \eqref{fteq} with $A$ replaced by $A'$ and with the choices $s=-1/q'$ and $\alpha_n=1$, $n=1,\dots, N$. Then
$\|f_t\|_{F^{-1/q'}_{p,q}}\lc N^{1/q}$.
By Proposition \ref{qlowerbd} there exist $t_1$, $t_2$ in $[0,1]$ so that
$\|T_{t_1} f_{t_2} \|_{F^{-1/q'}_{p,q}}\gc N.$
Hence
$\|T_{t_1}  \|_{F^{-1/q'}_{p,q} \to F^{-1/q'}_{p,q}}\ge c_{p,q}
N^{1/q'} .$
Now let
$E^{\pm}$ be as in \eqref{Epm}. Then
$\max_\pm \| P_{E^\pm}\|_{F^{-1/q'}_{p,q} \to F^{-1/q'}_{p,q}}\ge
\frac{c_{p,q}}{2}N^{1/q'}.$  Thus $\cG (F^{-1/q'}_{p,q} , A)
\ge \cG (F^{-1/q'}_{p,q} , A')
\gc N^{1/q'}$ and hence
$\gamma_* (F^{-1/q'}_{p,q} ;\Lambda)\gc (\log \Lambda)^{1/q'}$.

\section{Lower bounds for the endpoint case}
\label{gammaupperstar}
In this section we prove the lower bounds in Theorem \ref{GXAbds}. The following
result provides a slightly sharper bound where a min is replaced by an average.


\begin{theorem}  \label{Zthm} Assume $\#A\ge 4^N$ and that  $4^N$
disjoint  intervals $I_\ka$, $\ka=1,\dots, 4^N$ are given such  that
the length  of $I_\ka$ is  $N$,  and such that
$I_\ka\cap \log_2A\neq \emptyset$.
Let
\Be\label{Zdef}  Z=
\frac{1}{4^{N}}\sum_{\kappa=1}^{4^N} \#(I_\ka\cap \log_2A)\,.
\Ee
Then,  for $q\le p<\infty$,
$$\cG(F^{-1/q'}_{p,q};A)  \ge c(p,q) N^{1-1/q}
Z^{1/q}.$$ \end{theorem}


\begin{proof}[Proof that Theorem  \ref{Zthm} implies  Theorem \ref{GXAbds}]
The upper bounds follow easily from Theorem \ref{upperlabounds}. For the lower bounds let  $A\subset\{2^n: n\ge 1\}$  be of large cardinality and let $N$ be such that
$8^{N-1}\le \#A \le 8^N$. Let $\underline \cZ(A)=Z$. Then we can find
$M_N$ disjoint  intervals
$$I_i= (n_i-3N, n_i+3N)$$  with  midpoints $n_i\in \log_2(A)$,
$i=1,\dots, M_N$  so that $M_N\ge 8^{N-1}/N $ and so that each $I_i$ contains at least $Z$ points in $\log_2(A)$. Each $I_i$ contains a subinterval $\widetilde I_i$ of length
$N$ which contains at least $Z/6$ points. This means that the hypothesis of Theorem \ref{Zthm} is satisfied, and we get
$\cG(F_{p,q}^{-1/q'};A)\gc c(p,q) N^{1/q'}(Z/6)^{1/q}$. Part b) of Theorem \ref{GXAbds} follows since $\#\log_2(A) \approx N$. Part a) follows by duality.
\end{proof}

\subsection*{Proof of Theorem \ref{Zthm}}


Let $b_\ka$ be the largest integer in $I_\ka$ and
\Be\fL=\{b_\ka+N: \ka=1,\dots, 4^N\}\,,\Ee
\begin{subequations}
\begin{align}
\label{fKM}\fA(\ka)&= \{j\in I_\ka: 2^j\in A\}\,, \\
\label {EM}
\cE(\ka)&=\{(j,\mu): j\in \fA(\ka), \,\,\mu\in 2^{j-b_\ka+N+2}\bbZ, 1\le
\mu<2^j\}\,,
\\
\label {Efull}
 \cE&= \bigcup_{\ka=1}^{4^N} \cE(\ka)\,.
\end{align}
\end{subequations}
Let further $\eta$ be as in \eqref{etanondeg} and
\Be \label{Hl}
H_l(x)=
\sum_{1\le \sigma\le N} 2^{-\sigma}
 \sum_{\substack{\rho\in \bbN:\\ 0<2^{2N+2-l}\rho<1}}
\eta(2^{l-\sigma}(x- 2^{2N+2-l}\rho))\,.
\Ee
Define, for $t\in [0,1]$
\Be \label{gt}f_t(x)= \sum_{l\in \fL} r_l(t) 2^{l/q'} H_l(x)\,.\Ee

\begin{lemma}\label{gtuplemma}  We have
$$\|f_t\|_{F^{-1/q'}_{p,q}} \le C(p,q) N^{1/q}\,.
$$
uniformly in $t$.
\end{lemma}

\begin{proof}
For $\sigma=1,\dots, N$ let  $$\fL(\sigma)=\{b_\ka+N-\sigma: \ka=1,\dots, 4^N\}.$$ Thus
the $\fL(\sigma)$ are disjoint sets, of cardinality $2^{2N}$ each.
Let $\{r_j\}_{j=1}^\infty$ be the system of Rademacher functions and define, for $t\in [0,1]$,
\begin{align*}
g_{\sigma,t}
&=\sum_{\ka=1}^{2^{2N}}
2^{(b_\ka+N-\sigma)/q'}
\sum_{\substack{\rho\in \bbN:\\ 0<2^{2N+2-l}\rho<1}}
  r_{l+\sigma}(t)  \eta(2^{b_\ka+N-\sigma}(x- 2^{N+2-b_\ka}\rho))
\\
&= \sum_{l\in \fL(\sigma)}
2^{l/q'}
\sum_{\substack{\rho\in \bbN:\\ 0<2^{2N+2-l}\rho<1}}
  r_{l+\sigma}(t)  \eta(2^{l}(x- 2^{2N+2-l-\sigma}\rho))
\end{align*}
so that $$f_t=\sum_{\sigma=1}^N 2^{-\sigma/q} g_{\sigma,t}\,.$$
We apply
 Proposition \ref{testfctprop} with the parameter $N$ replaced by $2N$ and
$m=2N-\sigma$. Clearly, the points
$2^{l-\sigma}2^{2N+2-l}\rho$ are then  $2^{m-l}$ separated.
By
inequality \eqref{betamineq} with
$\beta_{2N-\sigma}=2^{-\sigma/q}$,
\[
\|f_t\|_{F^{-1/q'}_{p,q}}\lc\Big(\sum_{\sigma=1}^N  (2^{-\sigma/q})^q
2^{\sigma-2N}\# (\fL(\sigma))\Big)^{1/q} \lc N^{1/q}\,.\qedhere
\]\end{proof}
\noindent Define for $t\in [0,1]$
$$\cT_tf(x)= \sum_{(j,\mu)\in \cE} r_j(t)  {2^j}\inn {f}{h_{j,\mu} }h_{j,\mu}(x).
$$
\begin{proposition} \label{Ttlowprop}
Let $q<p<\infty$. Then
there is $c(p,q) >0$ such that for large $N$
\Be \label{Ttlow}
\Big(\int_0^1\int_0^1 \big\|\cT_{t_1} f_{t_2}  \big\|^q_{F^{-1/q'}_{p,q}} dt_2 dt_1\Big)^{1/q}
\ge c(p,q) N Z^{1/q}\,.
\Ee
\end{proposition}
\begin{proof}
By \eqref{localmeans} and H\"older's inequality
it  suffices to show
\Be\label{locallowerbdqqlim}
\Big(\int_0^1\int_0^1\Big\|\Big(\sum_\ka\sum_{k\in \fA(\ka))}2^{kq/q'}|\psi_k *\cT_{t_1}f_{t_2}|^q\Big)^{1/q}\Big\|_q^q dt_1 dt_2 \Big)^{1/q}
\gc NZ^{1/q}\,.\Ee
If we interchange variables and apply Khinchine's inequality then
\eqref{locallowerbdqqlim} follows from
\Be\label{sqfctlow}
\begin{split}
&\Big(\sum_\ka\sum_{k\in \fA(\ka)} 2^{kq/q'}
\Big\|\Big(\sum_{j}\sum_{l\in \fL} \\
&\hspace{1cm}\Big|
\sum_{\substack{\mu:\\ (j,\mu) \in \cE(\ka)}}
2^j \inn{2^{l/q'}H_l}{h_{j,\mu}}\psi_k*h_{j,\mu}
\Big|^2\Big)^{1/2} \Big\|_q^q\Big)^{1/q}
\gc NZ^{1/q}\,.
\end{split}
\Ee
We drop all terms with $(j,l)\neq (k,b_\ka+N)$ and  see that the  left hand side
of \eqref{sqfctlow} is bounded below by
\Be \label{needlowrboundfor}
\Big(\sum_\ka\sum_{k\in \fA(\ka)} 2^{kq/q'}
\Big\|\sum_{\substack{\mu:\\ (k,\mu) \in \cE(\ka)}}
2^{(b_\ka+N)/q'}  2^k \inn{H_{b_\ka+N}}{h_{k,\mu}}
\psi_k*h_{k,\mu}
 \Big\|_q^q\Big)^{1/q}\,.
\Ee
Let $J_{k,\mu}$ be as in \eqref{lowerboundforconv}.
With  $$\zeta_{\ka, \sigma,\rho}(x)=
\eta(2^{b_\ka+N-\sigma}(x- 2^{N+2-b_\ka}\rho))$$ we have
$$
  \inn{H_{b_\ka+N}}{h_{k,\mu}}=
\sum_{1\le\sigma\le N} 2^{-\sigma}\sum_{\substack{\rho:\\0<2^{N+2-b_\ka}\rho<1}}
 \inn{\zeta_{\ka, \sigma,\rho}}{h_{k,\mu}}.
$$ Recall that by \eqref{EM} we only consider $\mu$ of the form
$\mu=\mu_n:= 2^{k-b_M+N+2} n$
 for $n\in \bbN$. For those $\mu$,
\begin{align*}
&2^k \inn{\zeta_{\ka, \sigma,\rho}}{h_{k,\mu_n}}
\\
&=2^k \int  \eta(2^{b_\ka+N-\sigma}(x- 2^{N+2-b_\ka}\rho)) h_{0,0}(2^kx-\mu_n) dx
\\
&=\int \eta( 2^{b_\ka+N-\sigma-k}u + 2^{b_\ka+N-\sigma-k}\mu_n -2^{2N+2-\sigma}\rho) h_{0,0} (u) du
\\
&=\int \eta( 2^{b_\ka+N-\sigma-k}u + 2^{2N+2-\sigma}(n-\rho)) h_{0,0} (u) du\,.
\end{align*}
For $k\in \fA(\ka)$ we have $2^{b_\ka+N-\sigma-k} \le \frac 14 2^{2N-\sigma+2}$ and
since $\eta$ is supported in $(2^{-5},2^5)$ we see that
$$
2^{k-\sigma} \inn{\zeta_{\ka, \sigma,\rho}}{h_{k,\mu_n}}
= \begin{cases}  2^{k-b_\ka-N} \int_0^{1/2} \eta(u) du  \text{ if } n=\rho\,,
\\0 \text{ if } n\neq \rho\,.\end{cases}
$$
Hence,
\Be
2^k \inn{H_{b_\ka+N}}{h_{k,\mu_n}} = N  2^{k-b_\ka-N} \int_0^{1/2} \eta(u) du\,.
\Ee
The intervals $J_{k,\mu_n}$
are disjoint.
Hence the expression \eqref{needlowrboundfor} is bounded below by
\begin{multline*}c \Big(\sum_\ka\sum_{k\in \fA(\ka)} 2^{kq/q'}\,\,\times\\
\sum_{\substack{n:\\0< 2^{k-b_\ka+N+2} n<2^k}} \int_{J_{k,\mu_n}} \Big|
2^{(b_\ka+N)/q'} N  2^{k-b_\ka-N} \int_0^{1/2} \eta(u) du\Big|^q
dx\Big)^{1/q}\,.
\end{multline*}
The measure of
$\cup_{n:0< 2^{k-b_\ka+N+2} n<2^k} J_{k,\mu_n}$ is $\approx 2^{b_\ka-N-k}$.
Hence, the last expression is bounded below by
\begin{align*}&c' \Big(\sum_\ka\sum_{k\in \fA(\ka)} 2^{kq/q'}2^{b_\ka-N-k} [
2^{(b_\ka+N)/q'} N  2^{k-b_\ka-N}]^q\Big)^{1/q}
\\
&\gc \Big(\sum_\ka\sum_{k\in \fA(\ka)} 2^{-2N}N^q\Big)^{1/q} \gc NZ^{1/q}\,.
\end{align*}
This proves \eqref{sqfctlow}
and completes the proof of the proposition.
\end{proof}

\begin{proof}[Proof of Theorem \ref{Zthm}, conclusion.]
 By Lemma
\ref{gtuplemma} and Proposition \ref{Ttlowprop} there exist $t_1$, $t_2$
in  $[0,1]$ such that
$\|f_t\|_{F^{-1/q'}_{p,q}}\lc N^{1/q}$ and
$\|T_{t_1} f_{t_2} \|_{F^{-1/q'}_{p,q}}\gc N Z^{1/q}.$
Hence
$$\|T_{t_1}  \|_{F^{-1/q'}_{p,q} \to F^{-1/q'}_{p,q}}\ge c_{p,q}
N^{1-1/q}Z^{1/q}. $$ As in the previous section, if
$E^{\pm}$ is  as in \eqref{Epm}. Then
$$\max_\pm \| P_{E^\pm}\|_{F^{-1/q'}_{p,q} \to F^{-1/q'}_{p,q}}\ge
\frac{c_{p,q}}{2}N^{1-1/q}Z^{1/q}.$$ Thus $\cG(F^{-1/q'}_{p,q} , A)
\gc N^{1/q'}Z^{1/q}$ .
\end{proof}


\section{Concluding remarks}\label{remarksect}
\subsection{\it }
It is possible to disprove the
unconditional basis property also in the case $q/(q+1)\leq p\leq 1$ and $s\geq
1/q$
via complex interpolation, see Figure \ref{fig2} above. Indeed, if $E$ is
a subset of the Haar system the (quasi-)norm of the corresponding projection
operator $P_E$ interpolates as follows
\Be\label{Interp}
    \|P_E\|_{F^s_{p,q} \to F^{s}_{p,q}} \lc \|P_E\|^{1-\theta}_{F^{s_0}_{p_0,q}
\to F^{s_0}_{p_0,q}}\cdot \|P_E\|^{\theta}_{F^{s_1}_{p_1,q}
\to F^{s_1}_{p_1,q}}
\Ee
with $1/p = (1-\theta)/p_0 + \theta/p_1$, $s = (1-\theta)s_0 + \theta s_1$\,.
Since $P_E:F^{s_0}_{p_0,q} \to F^{s_0}_{p_0,q}$ is $O(1)$ we obtain the relation
$$
   \|P_E\|_{F^{s_1}_{p_1,q} \to F^{s_1}_{p_1,q}} \gc
\|P_E\|^{1/\theta}_{F^s_{p,q} \to F^{s}_{p,q}}\,.
$$
Choosing $(1/p,s)$ as in Figure \ref{fig2} below we obtain large quasi-norms in the
given quasi-Banach region. For the endpoint case the argument has to be
modified by interpolating along the $s=1/q$ line. Putting in \eqref{Interp}
the upper bounds from Theorem \ref{upperlabounds} and the lower bounds
from Theorem \ref{Zthm} we obtain ``large'' projections norms. These
observations show that the shaded region displayed in Figure \ref{fig2} is the
correct one for Hardy-Sobolev spaces $F_{p,2}^s$ on the real line, see also
\cite{triebel-bases}, \S 2.2.3, Page 82.
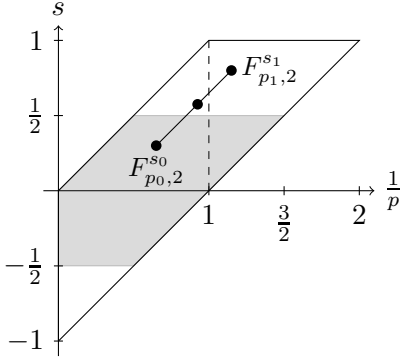
\begin{figure}[h]
 \begin{center}
\begin{tikzpicture}[scale=2]
\draw[->] (-0.1,0.0) -- (2.1,0.0) node[right] {$\frac{1}{p}$};
\draw[->] (0.0,-1.1) -- (0.0,1.1) node[above] {$s$};

\draw (1.0,0.03) -- (1.0,-0.03) node [below] {$1$};
\draw (2.0,0.03) -- (2.0,-0.03) node [below] {$2$};
\draw (1.5,0.03) -- (1.5,-0.03) node [below] {$\frac{3}{2}$};
\draw (0.03,1.0) -- (-0.03,1.00) node [left] {$1$};
\draw (0.03,.5) -- (-0.03,.5) node [left] {$\frac{1}{2}$};
\draw (0.03,-.5) -- (-0.03,-.5) node [left] {$-\frac{1}{2}$};
\draw (0.03,-1.0) -- (-0.03,-1.00) node [left] {$-1$};

\draw[dashed] (1.0,0.0) -- (1.0,1.0);
\draw[fill=black!70, opacity=0.2] (0.0,-.5) -- (0.0,0.0) -- (.5,.5) -- (1.5,0.5)
-- (1.0,0.0) -- (.5,-.5) -- (0.0,-.5);
\draw (0.0,-1.0) -- (0.0,0.0) -- (1.0,1.0) -- (2.0,1.0) -- (1.0,0.0) --
(0.0,-1.0);
\draw (0.65,0.3) -- (1.15,0.8);
\fill (1.15,0.8) circle[radius=1pt] node [right] {\small
$F^{s_1}_{p_1,2}$} ;
\fill (0.65,0.3) circle[radius=1pt] node [below] {\small
$F^{s_0}_{p_0,2}$} ;
\fill (0.925,0.575) circle[radius=1pt];
\draw (0.15,-0.65) node {};
\end{tikzpicture}
\caption{The Haar basis in Hardy-Sobolev spaces, complex
interpolation}\label{fig2}
\end{center}
\end{figure}

\subsection{\it }
Concerning the
spaces $F^s_{p,q}(\bbR^n)$ we expect a
similar picture as in Figure \ref{fig1}. However, for the quasi-Banach situation
there will be a $n$-dependence, see \cite{triebel-bases}, \S 2.3.2, Page 94.
\subsection{\it }
The corresponding problem for the Faber basis, i.e., the
family of hat functions that are integrals of associated Haar functions (\cf.
\cite[Chapt.\ 3]{triebel-bases}), can be derived from
the results in this paper. Due to the shift of
regularity of the basis functions there will be corresponding shifts in the
parameter domain (shaded region), cf.  Figures \ref{fig1}, \ref{fig2} together
with \cite{triebel-bases}, \S 3.1.2, Page 127.


\subsection{\it }The proofs in this paper of the existence of
 projection operators with large norm are probabilistic. It is also
 possible to explicitly construct subsets of the Haar system for
which the corresponding projections have large operator  norms. This is
done in
the subsequent paper \cite{sudet}.


\subsection{\it } It will be shown in a forthcoming paper \cite{gsu} that
there are
suitable enumerations of $\sH$ which  form  a Schauder basis of $L_{p}^s$, for
$-1/p'<s<1/p$.
This result has also   extensions to $F^s_{p,q}$ spaces.

\newpage
\end{document}